\documentclass[11pt,a4paper]{article}

\usepackage{dsfont}
\usepackage{amsmath,amsthm,amssymb}
\usepackage{hyperref}
\usepackage{aligned-overset}
\usepackage[a4paper,top=3cm,bottom=3cm,left=3cm,right=3cm,marginparwidth=1.75cm]{geometry}
\usepackage{enumitem}
\setenumerate[1]{label=(\arabic*)}
\makeatletter
\let\@fnsymbol\@arabic
\makeatother
\newtheorem{theorem}{Theorem}
\newtheorem{theoremintro}{Theorem}

\newtheorem{prop}[theorem]{Proposition}
\newtheorem{lemma}[theorem]{Lemma}
\newtheorem{corollary}[theorem]{Corollary}
\newtheorem{corollaryintro}[theoremintro]{Corollary}

\theoremstyle{definition}
\newtheorem{definition}[theorem]{Definition}
\newtheorem{notation}[theorem]{Notation}
\newtheorem{example}[theorem]{Example}
\theoremstyle{remark}
\newtheorem{remark}[theorem]{Remark}
\numberwithin{theorem}{section}
\numberwithin{equation}{section}
\newcommand{\overbar}[1]{\mkern 1.5mu\overline{\mkern-1.5mu#1\mkern-1.5mu}\mkern 1.5mu}

\newcommand{\R}{\mathbb{R}}
\newcommand{\N}{\mathbb{N}}
\newcommand{\op}{\operatorname}

\newcommand{\Z}{\mathbb{Z}}
\newcommand{\C}{\mathbb{C}}
\renewcommand{\tilde}{\widetilde}

\def\acts{\curvearrowright}
\newcommand{\e}{\varepsilon}
\newcommand{\ssubset}{\subset\joinrel\subset}

\title{Equivariant $\mathcal{Z}$-stability for single automorphisms on simple $C^*$-algebras with tractable trace simplices}
\author{Lise Wouters\thanks{Department of Mathematics, KU Leuven, 3001 Leuven (Belgium), lise.wouters@kuleuven.be.\newline Funding: This work was partially supported by long term structural funding – Methusalem grant of the Flemish Government and supported by the Research Foundation Flanders (PhD-grant 11B6620N fundamental research).
}}
\date{\vspace{-5ex}}

\begin{document}
\maketitle
\textbf{Abstract.}
Let $A$ be an algebraically simple, separable, nuclear, $\mathcal{Z}$-stable $C^*$-algebra for which the trace space $T(A)$ is a Bauer simplex and the extremal boundary $\partial_e T(A)$ has finite covering dimension. We prove that each automorphism $\alpha$ on $A$ is cocycle conjugate to its tensor product with the trivial automorphism on the Jiang--Su algebra. At least for single automorphisms this generalizes a recent result by Gardella--Hirshberg--Vaccaro. 
If $\alpha$ is strongly outer as an action of $\Z$, we prove it has finite Rokhlin dimension with commuting towers. As a consequence it tensorially absorbs any automorphism on the Jiang--Su algebra. 

\section*{Introduction}
\addcontentsline{toc}{section}{Introduction}
Over the past few years, enormous advancements have been made in the Elliott program (see for example \cite{ElliotGongLinNiu,  TikuisisWhiteWinter, GongLinNiu_rev, GongLinNiu_rev2, ElliotGongLinNiu20, GongLin, GongLin20,Schafhauser}), leading to a significant improvement in our knowledge about simple, separable, nuclear $C^*$-algebras and their classification. A natural next step to further increase our understanding of these objects is to study their symmetries, which is why recently a lot of interest among $C^*$-algebraists goes out to investigating the structure and classification of actions of amenable
groups on simple, separable, nuclear $C^*$-algebras in the scope of the Elliott program.
This reflects a more general pattern that occurs throughout the study of operator algebras: after succesfully classifying a class of interesting objects by means of a functorial invariant, a natural follow-up question becomes to study their innate symmetries and to investigate to which extent the functorial invariant helps to classify these. A prime historical example of this phenomenon is given by the Connes--Haagerup classification of injective factors \cite{Connes76,Haagerup}, which involved a classification of cyclic group actions \cite{Connes73,Connes75}. In turn, the increased interest in group actions on von Neumann algebras led to the classification of actions of countable amenable groups on injective factors, see \cite{Connes77, Jones, Ocneanu, SutherlandTakesaki, KawahigashiSutherlandTakesaki, KatayamaSutherlandTakesaki, Masuda, Masuda13}.
This paper aims to contribute to the ongoing research on classifying amenable group actions on classifiable $C^*$-algebras by tackling a related question, namely that of the potential scope of such a result. More precisely, the question is whether one can expect to classify all actions of amenable groups on classifiable $C^*$-algebras up to cocycle conjugacy by a computable invariant, or if one needs to impose an additional regularity criterion. 

In the context of the Elliott program, the analogous question received a lot of attention after the ermergence of certain exotic counterexamples \cite{Villadsen, Rordam03, Toms} made clear that it was impossible to achieve the original goal of classifying all simple, separable, nuclear $C^*$-algebras by the Elliott invariant (an invariant consisting of $K$-theory, traces and the natural interaction between these). The fine structure of these counterexamples was too complicated and could only be captured by a much finer invariant such as the Cuntz semigroup. It became clear that in contrast to amenable factors, which are automatically well-behaved, there is a dichotomy for simple, separable, nuclear $C^*$-algebras, separating those that can be classified by the Elliott invariant and their much wilder counterparts. In turn, this led to the study of several regularity properties that could possibly distinguish these two types of $C^*$-algebras and grant access to classification. The famous Toms--Winter conjecture \cite{ElliottToms,Winter, WinterZacharias10, Winter12} states that for non-elementary, simple, separable, nuclear $C^*$-algebras three of these at first sight totally different properties are equivalent. It is almost an established theorem by now \cite{Rordam04,MatuiSatoStrictComparison,MatuiSatoDecompositionRank,Tikuisis,SatoWhiteWinter, BBSTWW,CETWW,CastillejosEvington}, missing only one implication in full generality.  Moreover, recent progress on this last implication shows that it holds true when assuming the $C^*$-algebra has either uniform property Gamma \cite{CETW_upG} or stable rank one and locally finite nuclear dimension \cite{Thiel}.
One especially noteworthy property appearing in the conjecture is $\mathcal{Z}$-stability. A $C^*$-algebra $A$ is called $\mathcal{Z}$-stable if it is isomorphic to the tensor product $A \otimes \mathcal{Z}$, where $\mathcal{Z}$ is the standard notation for the Jiang--Su algebra \cite{JiangSu}. Since $\mathcal{Z}$ has unique trace and the same $K$-theory as $\mathbb{C}$, it can be thought of as an infinite-dimensional analogue of $\mathbb{C}$, and the Elliott invariant cannot distinguish between them. This carries over to tensor products, meaning that the Elliott invariant cannot distinguish a $C^*$-algebra $A$ from the tensor product $A \otimes \mathcal{Z}$.\footnote{This statement is certainly true according to the most recent ongoing approach to classification, where the order of the $K_0$-group that was previously included in the invariant is now left out. The statement remains true under some mild conditions when the order is included.} As shown by Toms \cite{Toms}, this is not a problem inherent only to the Elliott invariant, as he was able to produce systematic examples of simple, separable, nuclear $C^*$-algebras $A$ that could not be distinguished from $A \otimes \mathcal{Z}$ by any reasonably computable invariant, even though $A \ncong A \otimes \mathcal{Z}$. Knowing this, it is natural to restrict classification attempts to $\mathcal{Z}$-stable $C^*$-algebras. By the combined efforts of many people involved in the earlier mentioned references, it is now proved that $\mathcal{Z}$-stability is also sufficient, in the sense that classification is possible for simple, separable, nuclear, $\mathcal{Z}$-stable $C^*$-algebras in the UCT-class.

In the case of group actions, a similar phenomenon occurs. A sufficiently manageable invariant for the category of $G$-$C^*$-dynamical systems will usually not be able to distinguish the action $\alpha\colon G \acts A$ from the action $\alpha \otimes \op{id}_\mathcal{Z}\colon G \acts A \otimes \mathcal{Z}$. In particular, this is the case for invariants like the action induced on the tracial simplex, and equivariant $KK$-theory \cite{Kasparov}. So when it comes to classifying $G$-$C^*$-dynamical systems up to cocycle conjugacy (see Definition~\ref{def:cc}), for the same reasons as before, the only actions we can expect to classify are the equivariantly $\mathcal{Z}$-stable ones, i.e.\ actions that are cocycle conjugate to their tensor product with the trivial action on $\mathcal{Z}$. In order to clarify the scope of the classification program for amenable group actions, it is therefore important to investigate whether this property holds automatically for all actions of amenable groups acting on classifiable $C^*$-algebras or not.

 Unlike for ordinary $\mathcal{Z}$-stability, there is evidence in the literature that suggests equivariant $\mathcal{Z}$-stability could indeed be automatic for  actions of amenable groups on simple, separable, nuclear, $\mathcal{Z}$-stable $C^*$-algebras.\footnote{Without amenability of the group this is not true in general, as is shown by natural counterexamples given in \cite{GardellaLupini}, based on the work of Jones \cite{Jones83}.}\ For the traceless case equivariant $\mathcal{Z}$-stability follows from equivariant Kirchberg--Phillips results by Izumi--Matui for poly-$\Z$ groups \cite{IzumiMatui,Izumi12,IzumiMatui19,IzumiMatui20}, Goldstein--Izumi for finite groups \cite{GoldsteinIzumi} and Szab\'o for general amenable groups \cite{EquivariantKP} that imply automatic absorption of the trivial action on $\mathcal{O}_\infty$. For actions on finite classifiable $C^*$-algebras, equivariant $\mathcal{Z}$-stability results have been obtained in \cite{MatuiSatoZstability,MatuiSatoZstabilityII,Sato19,GardellaHirshberg, GardellaHirshbergVaccaro}, all relying on additional assumptions about the tracial simplex. For example, the most recent and general of these results by Gardella--Hirshberg--Vaccaro \cite{GardellaHirshberg, GardellaHirshbergVaccaro} shows that actions of amenable groups on unital classifiable $C^*$-algebras are automatically $\mathcal{Z}$-stable under the additional assumptions that the tracial state space $T(A)$ of the $C^*$-algebra is a non-empty Bauer simplex\footnote{A \emph{Bauer simplex} is  a Choquet simplex with compact extremal boundary.}, the extremal boundary $\partial_e T(A)$ has finite covering dimension, and the action that $G$ induces on $\partial_e T(A)$ factors through a finite quotient group.\footnote{The exact assumptions on the induced action $G \acts \partial_e T(A)$ in \cite{GardellaHirshberg, GardellaHirshbergVaccaro} are a bit more subtle and differ slightly between the two articles.} 
 
The restrictions on the tracial simplex in this kind of results stem from the limitations of the current tools that are available to tackle these types of problems. To understand this a bit better, one needs to look at the nature of the problem. As explained in Section \ref{section:reduct_1}, in the presence of traces an equivariant version of a well-known argument due to Matui--Sato \cite{MatuiSatoStrictComparison} shows that for an action $\alpha\colon G \acts A$ on a classifiable $C^*$-algebra $A$, equivariant $\mathcal{Z}$-stability is equivalent to the existence of a unital $*$-homomorphisms $M_n \rightarrow (A^\omega \cap A')^{\alpha^\omega}$ for all $n \in \N$ into the fixed point algebra of the action that $\alpha$ induces on the uniform tracial central sequence algebra (see Definition \ref{definition:ultrapowers}). The uniform tracial ultrapower $A^\omega$ is a central object of interest for finite $C^*$-algebras that can be considered as an analogue of ultrapowers for tracial von Neumann algebras and takes into account all traces simultaneously. 
The properties of this object ensure that the existence of the $*$-homomorphism above can be checked by finding certain elements in $A$ that satisfy a finite number of conditions approximately, such as being central and fixed by $\alpha$ in the uniform tracial 2-norm. In this way, the equivariant $\mathcal{Z}$-stability problem turns into a tracial problem that can be solved by a `local-to-global' approach. The terminology refers to the desired `global' result that must hold uniformly over all traces, which one aims to derive from known `local' results, i.e.\ results that hold only with respect to one trace at the time by virtue of results in von Neumann algebra theory. In \cite{GardellaHirshberg}, Gardella--Hirshberg used the framework of $W^*$-bundles (see Section \ref{section:W*-bundles}) to make this approach work and obtain their result, but this only works when the tracial state space is a Bauer simplex. The further assumptions on the tracial state space are needed to make their local-to-global argument work. In this paper we still use the same $W^*$-bundle techniques to build further on their work, and are able to improve on their result for single automorphisms by eliminating the requirement that the action induced on $\partial_e T(A)$ must factor through a finite quotient group: 
 
\begin{theoremintro}\label{theorem:intro_main1}
Let $A$ be an algebraically simple, separable, nuclear, $\mathcal{Z}$-stable $C^*$-algebra. Suppose that $T(A)$ is a non-empty Bauer simplex and that $\partial_e T(A)$ is finite-dimensional. Then any automorphism $\alpha \in \op{Aut}(A)$ satisfies $\alpha \simeq_{\mathrm{cc}} \alpha \otimes \op{id}_\mathcal{Z}$.
\end{theoremintro} 
 
For a unital $C^*$-algebra $A$ the algebraic simplicity assumption appearing in Theorem \ref{theorem:intro_main1} is automatically satisfied if $A$ is simple, but this is not necessarily the case for non-unital $A$. Algebraic simplicity guarantees there are no unbounded tracial weights, and together with the other assumptions on the tracial state space this allows us to carry out the argument using the same framework as in the unital case with minimal adaptions and no conceptual differences. For the moment it is unclear how to prove Theorem \ref{theorem:intro_main1} without this assumption, since it is unclear what the correct framework would be in the presence of unbounded traces. It is, however, important to realize that this really entails a loss of generality. In the dynamics-free context, one can sometimes reduce questions about structural properties of $C^*$-algebras to the algebraically simple case by considering carefully chosen hereditary subalgebras, as is done in \cite{CastillejosEvington}. However, this type of reduction argument cannot work in general for group actions. For example, when considering trace-scaling automorphisms on stable $C^*$-algebras it is not always possible to find a hereditary substructure resembling an automorphism on an algebraically simple $C^*$-algebra. 
 
In order to prove Theorem \ref{theorem:intro_main1} relying on the work of Gardella--Hirshberg, one needs to find a way to handle possibly infinite orbits for the induced action on the trace space. As an intermediate result, we prove the theorem in case the induced action is free. In fact, we prove something stronger (see Theorem \ref{theorem:main_result_free}) from which the next result follows. 
\begin{theoremintro}\label{theorem:intro_free}
Let $A$ be an algebraically simple, separable, nuclear, $\mathcal{Z}$-stable $C^*$-algebra. Suppose that $T(A)$ is a non-empty Bauer simplex and that $\partial_e T(A)$ is finite-dimensional. Take a countable discrete group $G$ for which all finitely generated subgroups are virtually nilpotent and take an action $\alpha\colon G \acts A$. If the action that $\alpha$ induces on $\partial_e T(A)$ is free, then  $\alpha \simeq_{\mathrm{cc}} \alpha \otimes \op{id}_\mathcal{Z}$.
\end{theoremintro}
Together with Theorem \ref{theorem:intro_main1} above, this theorem provides further partial verification of the conjecture that equivariant $\mathcal{Z}$-stability is automatic for actions of amenable groups on classifiable $C^*$-algebras (cfr.\ \cite[Conjecture A]{EquivariantSI}). 
The proof of Theorem B relies on the results of \cite{SzaboWuZacharias, RokhlinDimension}. A crucial ingredient is the existence of an equivariant inclusion $C(\partial_e T(A)) \subset A^\omega\cap A'$, which allows one to perform an averaging argument using Rokhlin towers originating from the system $G \acts C(\partial_e T(A))$ to take a map $M_n \rightarrow A^\omega\cap A'$ and construct from it a map $M_n \rightarrow (A^\omega \cap A')^{\alpha^\omega}$. 

If one is interested in $\mathcal{Z}$-stability of the crossed product $A \rtimes_\alpha G$ (a weaker conclusion) instead of equivariant $\mathcal{Z}$-stability, then the recent work \cite{GGNV} shows that some of the assumptions in Theorem B above can be weakened. Taking inspiration from the present work, the authors of \cite{GGNV} show how the uniform Rohklin property of $G \acts \partial_eT(A)$ (first introduced in \cite{Niu}) allows one to deduce $\mathcal{Z}$-stability of $A \rtimes_\alpha G$ in the context of Theorem B, but for general countable amenable groups, without needing the finite-dimensionality assumption on $\partial_e T(A)$ in case of free minimal $\Z^d$-actions.

In the general case of single automorphisms, combining the ideas from the proof of Theorem B with the results of Gardella--Hirshberg in \cite{GardellaHirshberg} allows us to derive Theorem \ref{theorem:intro_main1}. By an observation that seems to originate in \cite{HirshbergWinterZacharias}, instead of finding one unital $*$-homomorphism $M_n \rightarrow (A^\omega \cap A')^{\alpha^\omega}$ it suffices to find two pairwise commuting c.p.c.\ order zero maps $\Psi_i\colon M_n \rightarrow (A^\omega \cap A')^{\alpha^\omega}$ for $i \in {1,2}$ that are jointly unital, in the sense that $\Psi_1(1) + \Psi_2(1)= 1$. Alternatively, the approximate version of this statement implies that it suffices to find c.p.c.\ order zero maps $\psi_1$ and $\psi_2$ from $M_n$ into the uniform tracial closure $\bar{A}^u$ (see Definition \ref{definition:uniform_tracial_closure}) satisfying the above properties approximately. This uniform tracial closure $\bar{A}^u$ is the $C^*$-algebra obtained from $A$ by adding limit points of bounded sequences that are Cauchy in the uniform tracial 2-norm. A lot of information about $A$ gets lost in this process, such as the $K$-theory, but $\bar{A}^u$ still captures the tracial properties of $A$. Moreover, when $A$ satisfies the assumptions in Theorem \ref{theorem:intro_main1}, the work of Ozawa \cite{Ozawa} shows that $\bar{A}^u$ has the structure of a trivial $W^*$-bundle with base space $\partial_e T(A)$ and fibers isomorphic to the hyperfinite II$_1$ factor $\mathcal{R}$. More specifically, as a $W^*$-bundle it is isomorphic to $C_\sigma(\partial_e T(A), \mathcal{R})$, the norm bounded and strong operator continuous functions from $\partial_e T(A)$ to $\mathcal{R}$. Thus, what really lies at the core of proving Theorem \ref{theorem:intro_main1}, is solving a  problem about automorphisms on the trivial bundle $\bar{A}^u \cong C_\sigma(\partial_e T(A), \mathcal{R})$. In order to obtain our result we take advantage of the specific way in which $\Z$ acts on the base space $\partial_e T(A)$, which allows us to divide this extremal boundary into two subsets: a set $X_1$ where all points have orbit size at most some fixed $N \in \N$, and its complement $X_2 = \partial_e T(A)\setminus X_1$. We will construct $\psi_1$ and $\psi_2$ such that they are approximately unital on the fibers corresponding to $X_1$ and $X_2$, respectively. This idea of dividing the space into a part with small and large periods to merge two different methods of proof draws inspiration from the proof of the main result in \cite{HirshbergWu}. The map $\psi_1$ can be obtained from first restricting the bundle $\bar{A}^u$ to $X_1$, then using the results of Gardella--Hirshberg (since the action induced on the base space of this restricted bundle has orbits uniformly bounded in size) and lifting this back to a c.p.c.\ order zero map into $\bar{A}^u$. For the map $\psi_2$ we apply a Rokhlin-type result from topological dynamics to the system $\Z \acts X_2$. Via the inclusion $C_0(X_2) \subset \mathcal{Z}(\bar{A}^u)$ into the center of the $W^*$-bundle we are able to obtain Rokhlin towers which allow us to average an approximate $*$-homomorphism $M_n \rightarrow \bar{A}^u$ in such a way that it lands approximately in the fixed point algebra of $\alpha$, at least if $N$ was chosen large enough. 

With some extra work, the same methods can be used to prove that  strongly outer $\Z$-actions have finite Rokhlin dimension with commuting towers. Rokhlin dimension is one of the various Rokhlin-type properties for group actions studied by $C^*$-algebraists, motivated by the importance of the Rokhlin lemma in ergodic theory and Rokhlin-type results in the classification of amenable group actions on von Neumann algebras. The Rokhlin property for $C^*$-algebras first appeared in the work of Herman--Jones \cite{HermanJones} and Herman--Ocneanu \cite{HermanOcneanu} about actions of cyclic groups on UHF-algebras, and its applications to classification for single automorphisms were later studied in greater generality by Kishimoto and various collaborators \cite{BKRS, Kishimoto95,BEK,Kishimoto96,EvansKishimoto,Kishimoto98,Kishimoto98b,EEK,BratelliKishimoto,Nakamura}.   
 The finite group case was studied by Izumi in \cite{IzumiI,IzumiII} (for unital $C^*$-algebras; the non-unital case was studied for example in \cite{GardellaSantiago}). As his research shows, aside from other possible obstructions like a lack of projections in the $C^*$-algebra, in the case of finite groups there are obstructions of K-theoretic nature that withhold a lot of actions from having the Rokhlin property. This led researchers to study two types of weakened versions of the property:

The first is the tracial Rokhlin property, 
 introduced by Phillips for finite groups in \cite{Phillips}, which allows that the sum of the projections appearing in the property has a small leftover in trace. This property has been used among other examples in the work of Phillips and his collaborators where they studied the structure of crossed products of irrational rotation algebras by certain actions of cyclic groups \cite{ELPW}. Since this version of the property is still of no use in the absence of non-trivial projections, a weakened version of this property, called the weak (tracial) Rokhlin property, was considered in \cite{Sato, HirshbergOrovitz, MatuiSatoZstability,MatuiSatoZstabilityII, GardellaHirshbergSantiago,GardellaHirshbergVaccaro}. In particular, its relation with strong outerness was studied further by Matui--Sato \cite{MatuiSatoZstability,MatuiSatoZstabilityII} and Gardella--Hirshberg--Vaccaro \cite{GardellaHirshbergVaccaro}.
 
The second is the notion of Rokhlin dimension, originally introduced by Hirshberg, Winter and Zacharias in \cite{HirshbergWinterZacharias} for finite groups and actions of $\Z$. Their definition allows for multiple towers instead of just the one as in the Rokhlin property, making the property of having finite Rokhlin dimension more common and flexible to work with. The definition was later extended to residually finite groups by Szab\'o--Wu--Zacharias \cite{SzaboWuZacharias}. In \cite{Liao,Liao17}, Liao showed that strongly outer $\Z^m$-actions on simple, separable, nuclear, $\mathcal{Z}$-stable $C^*$algebras have finite Rokhlin dimension under the assumption that the tracial state space is a Bauer simplex with a finite-dimensional extremal boundary that is fixed by the action. This was generalized by Gardella--Hirshberg--Vaccaro \cite{GardellaHirshbergVaccaro} to actions of amenable groups, keeping the same assumptions on the trace simplex except that the induced action on the extremal boundary is allowed to have finite orbits with uniformly bounded cardinality.

 From the beginning, Hirshberg, Winter and Zacharias also introduced a stronger version of the notion, namely with commuting towers. Here we prove the following result, which in the unique trace case is a special case of \cite[Theorem 2.14]{ActionsTorsionFree}:
\begin{theoremintro}\label{theorem:intro_main2}
Let $A$ be an algebraically simple, separable, nuclear, $\mathcal{Z}$-stable $C^*$-algebra. Suppose that $T(A)$ is a non-empty Bauer simplex and that $\partial_e T(A)$ is finite-dimensional. Then any strongly outer action $\alpha\colon \Z \acts A$ satisfies $\op{dim}_\mathrm{Rok}^c(\alpha) \leq 2$.
\end{theoremintro}

The link with Theorem \ref{theorem:intro_main1} becomes apparent when one realizes that in presence of equivariant property (SI), having finite Rokhlin dimension can also be translated to a more tractable tracial statement, see Section \ref{section:reduct_2}. More precisely, in order to show that an action $\alpha\colon \Z \acts A$ has finite Rokhlin dimension it suffices to prove that for each unitary representation $\nu\colon \Z \rightarrow \mathcal{U}(M_n)$ there exists a unital equivariant embedding
$(M_n, \op{Ad}(\nu)) \rightarrow (A^\omega \cap A', \alpha^\omega)$. The strategy to obtain these embeddings is completely similar to the strategy explained above for finding maps $M_n \rightarrow (A^\omega \cap A')^{\alpha^\omega}$. From the results of \cite{RokhlinDimension}, which established the connection between finite Rokhlin dimension with commuting towers and the absorption of strongly self-absorbing model actions, we obtain the following corollary: 
\begin{corollaryintro}\label{corollary:intro}
Let $A$ be an algebraically simple, separable, nuclear, $\mathcal{Z}$-stable $C^*$-algebra. Suppose that $T(A)$ is a non-empty Bauer simplex and that $\partial_e T(A)$ is finite-dimensional. Let $\mathcal{D}$ be a strongly self-absorbing $C^*$-algebra such that $A \cong A \otimes \mathcal{D}$. Then for any strongly outer action $\alpha\colon \Z \acts A$ and any $\delta\colon \Z \acts \mathcal{D}$ it holds that $\alpha \simeq_{\mathrm{cc}} \alpha \otimes \delta$.
\end{corollaryintro}

A previous result of this type appeared in \cite[Theorem 3.1]{ActionsTorsionFree}, where this was proved for a more general class of torsion-free elementary amenable groups under the assumption that $A$ has unique trace. This was used to prove that for any group $G$ in this class and any strongly self-absorbing $C^*$-algebra $\mathcal{D}$, there exists a unique strongly outer $G$-action on $\mathcal{D}$ up to cocycle conjugacy \cite[Corollary 3.3]{ActionsTorsionFree}. The result obtained here may prove useful to further study the structural properties of more general strongly outer actions. By choosing the right model action, one can for example obtain that these actions have the weak Rokhlin property mentioned above. As suggested by the role played by the Rokhlin type lemma in the work of Ocneanu \cite{Ocneanu}, when an action has the weak Rokhlin property this could prove a useful tool to classify the induced action on the uniform tracial closure. Furthermore, the result could also serve as a basis for obtaining Rokhlin-like regularity properties for strongly outer actions of more general amenable groups in future work.
 
This paper is structured as follows: we start with a preliminary section to recall some of the basic concepts, particularly ultrapowers and the uniform tracial completion of a given $C^*$-algebra. In Sections \ref{section:reduct_1} and \ref{section:reduct_2} we show that in order to prove Theorems \ref{theorem:intro_main1} and \ref{theorem:intro_main2} it suffices to solve a more tractable tracial problem. In Section \ref{section:syst_order_zero} we explain how the problem can be reduced further to finding systems of order zero maps with pairwise commuting ranges. The major technical work is done in Section \ref{section:W*-bundles}, where we prove the main results in the more general and abstract $W^*$-bundle setting. As an application, we derive the main results as they are stated above in Section \ref{section:main}.   
\section{Preliminaries}

\begin{notation}Throughout this paper, we will use the following notations and conventions:
\begin{itemize}[itemsep=2pt]
\item $\omega$ denotes a fixed free ultrafilter on $\N$. 
\item $G$ denotes a countable discrete group.
\item $A$ and $B$ denote $C^*$-algebras. The symbols $\alpha$ and $\beta$ will be used to denote single automorphisms or actions on $C^*$-algebras.
\item $\mathcal{M}(A)$ denotes the multiplier algebra of the $C^*$-algebra $A$.
\item If $\alpha\colon G \acts A$ is an action, then $A^\alpha$ denotes the fixed-point algebra. The orbit of $x \in A$ will be denoted by $G \cdot x$. 
\item If $F$ is a finite subset inside another set $M$, we often denote this by $F \ssubset M$.
\item $\mathcal{Z}$ is the standard notation for the so-called Jiang--Su algebra \cite{JiangSu}. We will assume the reader is familiar with this $C^*$-algebra. 
\item We will also assume the reader is familiar with c.p.c.\ order zero maps, see \cite{TomsWinter}.
\end{itemize}
\end{notation}

\begin{remark}
In most statements in this paper, the relevant $C^*$-algebras are assumed to be algebraically simple. This is automatically the case for unital simple $C^*$-algebras, but not for non-unital ones. If $A$ is an algebraically simple $C^*$-algebra, it does not admit any non-trivial unbounded tracial weights on $A_+$. Therefore, it makes sense to only consider tracial states, which we will refer to simply as \emph{traces}. This also explains the reason behind our algebraic simplicity assumption, since in the presence of unbounded traces it is unclear what the correct framework to treat the problems studied in this paper would be.  The collection of all traces will be denoted by $T(A)$ and will be equipped with the weak-* topology it gets from $A^*$.  
Furthermore we will always assume $T(A)$ to be compact. Again, this assumption is redundant for unital $C^*$-algebras, but is a genuine assumption in the non-unital case. Compactness of $T(A)$ guarantees the tracial state space has the structure of a Choquet simplex. We will denote the extremal boundary by $\partial_eT(A)$.
\end{remark}

We also recall some necessary definitions:

\begin{definition}\label{def:cc}
Let $\alpha\colon G \acts A$ be an action of a countable discrete group on a $C^*$-algebra.
\begin{enumerate}
\item A map $w\colon G \rightarrow \mathcal{U}(\mathcal{M}(A))$ is called an \emph{$\alpha$-1-cocycle}, if it satisfies $w_g \alpha_g(w_h) =w_{gh}$ for all $g, h \in G$. In this case, it induces a new action $\alpha^w\colon G \acts A$ defined by $\alpha_g^w = \op{Ad}(w_g) \circ \alpha_g$. This action is called a \emph{cocycle perturbation} of $\alpha$. 
\item Let $\beta\colon G \acts B$ be another action on a $C^*$-algebra. The actions $\alpha$ and $\beta$ are called \emph{cocycle conjugate} (denoted by $\alpha \simeq_{\mathrm{cc}} \beta$) if there exists an isomorphism $\phi\colon A \rightarrow B$ such that $\phi \circ \alpha \circ \phi^{-1}$ is a cocycle perturbation of $\beta$. 
\end{enumerate}

\begin{definition}\label{def:ssa}
Let $\mathcal{D}$ be a separable, unital $C^*$-algebra. 
\begin{itemize}
\item \cite[Definition 1.3]{TomsWinter} $\mathcal{D}$ is called \emph{strongly self-absorbing} if $\mathcal{D} \ncong \C$ and $\op{id}_\mathcal{D} \otimes 1_\mathcal{D}\colon \mathcal{D} \rightarrow \mathcal{D} \otimes \mathcal{D}$ is approximately unitarily equivalent to an isomorphism.\footnote{As in \cite{TomsWinter}, strong self-absorption is a priori defined using the minimal $C^*$-algebraic tensor product. As is shown there, strong self-absorption automatically implies nuclearity of the $C^*$-algebra, so the choice of tensor product turns out to be irrelevant afterwards.}
\item \cite[Definition 5.3]{CategoricalFramework} An action $\gamma\colon G \acts \mathcal{D}$ is called \emph{strongly self-absorbing} if there exists a sequence of unitaries $(u_n)_{n \in \N} \in \mathcal{D} \otimes \mathcal{D}$ such that the sequence of maps $\op{Ad}(u_n)\circ (\op{id}_\mathcal{D} \otimes 1_\mathcal{D} )\colon \mathcal{D} \rightarrow \mathcal{D}\otimes \mathcal{D}$ converges pointwise to a $*$-isomorphism (i.e.\ $\op{id}_\mathcal{D} \otimes 1_\mathcal{D}$ is approximately unitarily equivalent to an isomorphism) and the map
 \[G \rightarrow \mathcal{U}(\mathcal{D} \otimes \mathcal{D})\colon g \mapsto u_n (\gamma \otimes \gamma)_g (u_n^*)\] converges pointwise to a cocycle.
\end{itemize}
\end{definition}
\begin{remark}
Strongly self-absorbing actions were originally defined differently by Szab\'o in \cite{SsaDynSyst}. The definition was changed in his later work \cite{CategoricalFramework}, after it became clear that the notion above was more natural. This new definition agrees with what where originally called ``semi-strongly self-absorbing" actions in references \cite{SsaDynSyst3}--\cite{EquivariantSI}.
\end{remark} 
\begin{remark}From Definition \ref{def:ssa} we see the following:
\begin{itemize}
\item
If $\gamma\colon G \acts \mathcal{D}$ is strongly self-absorbing, then it is cocycle conjugate to $\gamma \otimes \gamma$. Let $\phi\colon \mathcal{D} \rightarrow \mathcal{D} \otimes \mathcal{D}$ denote the isomorphism that is approximately unitarily equivalent to $\op{id}_{\mathcal{D}} \otimes 1_{\mathcal{D}}$, and let $w$ denote the cocycle obtained from the sequence of unitaries realizing the approximate unitary equivalence. Then $\phi \circ \gamma_g \circ \phi^{-1} = \op{Ad}(w_g) \circ (\gamma \otimes \gamma)_g$ for all $g \in G$.
\item The trivial action on a separable, unital $C^*$-algebra is strongly self-absorbing if and only if the $C^*$-algebra itself is strongly self-absorbing. 
\end{itemize}
\end{remark}

\begin{definition} Let $A$ be a separable $C^*$-algebra and $\mathcal{D}$ a strongly self-absorbing $C^*$-algebra.
\begin{itemize}
\item $A$ is called \emph{$\mathcal{D}$-stable} if $A \cong A\otimes \mathcal{D}$. 
\item \cite[Remark following Theorem 2.6]{SsaDynSyst2} Let $\alpha \colon G \acts A$ be an action of a countable discrete group $G$ and let $\gamma\colon G \acts \mathcal{D}$ be a strongly self-absorbing action. The action $\alpha$ is called \emph{$\gamma$-absorbing} or \emph{$\gamma$-stable} if $\alpha \simeq_{\mathrm{cc}} \alpha \otimes \gamma$. In the special case that $\gamma$ is the trivial $G$-action, $\alpha$ is called \emph{equivariantly $\mathcal{D}$-stable}.
\end{itemize}
\end{definition}
 
\end{definition}
\begin{definition}
Let $A$ be a $C^*$-algebra. Each non-empty set $X \subset T(A)$ gives rise to a seminorm $\|\cdot\|_{2,X}$ on $A$, defined by
\[\|a\|_{2,X} := \sup_{\tau \in X}\tau (a^*a)^{1/2}.\]
The seminorm $\|\cdot\|_{2,T(A)}$ is also denoted by $||\cdot\|_{2,u}$. This is a norm if and only if for all non-zero $a \in A$ there exists some $\tau \in T(A)$ such that $\tau(a^*a) >0$, so in particular if $A$ is algebraically simple with non-empty trace space. 
\end{definition}
\begin{definition}[{\cite[Definition 1.1]{Kirchberg}, \cite[Definition 4.3]{KirchbergRordam}, \cite[1.3]{CETWW}}]\label{definition:ultrapowers}
Let $A$ be a $C^*$-algebra with an action $\alpha\colon G \acts A$ of a countable discrete group.
\begin{enumerate}
\item The \textit{ultrapower} of $A$ is defined as
\[A_\omega := \ell^\infty(A)/\{(a_n)_{n \in \N} \in \ell^\infty(A) \colon \lim_{n \rightarrow \omega} \|a_n\| = 0\}.\]
\item Pointwise application of $\alpha$ on representing sequences induces an action on the ultrapower, which we will denote by $\alpha_\omega\colon G \acts A_\omega$.
\item There is a natural inclusion $A \subset A_\omega$ by identifying an element of $A$ with its constant sequence. Define
\[A_\omega \cap A' := \{x \in A_\omega\mid [x,A]=0\} \quad \text{and}\quad A_\omega \cap A^\perp := \{x \in A_\omega \mid xA=Ax=0\}.\]
The quotient
\[F_\omega(A) := (A_\omega \cap A')/(A_\omega \cap A^\perp).\]
is called the \emph{(corrected) central sequence algebra}. For unital $A$ this equals $A_\omega \cap A'$.
\item Since $A$ is $\alpha_\omega$-invariant, so are $A_\omega \cap A'$ and $A_\omega \cap A^\perp$. Thus, $\alpha_\omega$ induces an action on $F_\omega(A)$, which we will denote by $\tilde{\alpha}_\omega\colon G \acts F_\omega(A)$.
\item When $A$ has no unbounded tracial weights and $T(A) \neq \emptyset$, the \emph{trace-kernel ideal} is defined by (abusing notation using representative sequences in $\ell^\infty(A)$ to denote elements in $A_\omega$)
\[J_A := \{ (a_n)_{n \in \N} \in A_\omega \mid \lim_{n \rightarrow \omega} \|a_n\|_{2,T(A)}= 0\}.\]
The \emph{uniform tracial ultrapower} is defined as
\[A^\omega := A_\omega/ J_A.\]
Whenever $\|\cdot\|_{2,T(A)}$ defines a norm on $A$, there also exists a canonical embedding of $A$ into $A^\omega$. Then $A^\omega \cap A'$ is called the \emph{uniform tracial central sequence algebra}. 
\item As $J_A$ is $\alpha_\omega$-invariant, there is an induced action on the uniform tracial ultrapower, which we will denote by $\alpha^\omega\colon G \acts A^\omega$. 
\end{enumerate}
\end{definition}

\begin{definition}
A sequence of traces $(\tau_n)_{n=1}^\infty$ on $A$ defines a trace on $A_\omega$ via
\[(a_n)_{n \in \N} \mapsto \lim_{n \rightarrow \omega} \tau_n(a_n).\]
A trace of this form is called a \emph{limit trace}. Clearly, every limit trace vanishes on $J_A$ and hence also induces a trace on $A^\omega$. We will denote the collection of these limit traces on both $A_\omega$ and $A^\omega$ by $T_\omega(A)$. Note that 
\[J_A =\{ x \in A_\omega \colon \|x\|_{2,T_\omega(A)} = 0\},\]
so in particular $\|\cdot\|_{2,T_\omega(A)}$ defines a norm on $A^\omega$. 
\end{definition}
\begin{prop}[{\cite[Proposition 1.1]{CETW_upG}}]\label{prop:tracial_ultrapower_unital}
Let $A$ be a separable $C^*$-algebra with $T(A)$ compact and non-empty. Then $A^\omega$ is unital. Moreover, if $(e_n)_{n\in \N}$ is an approximate unit for $A$, then $\|e_n - 1_{A^\omega}\|_{2, T_\omega(A)} \rightarrow 0$.
\end{prop}
\begin{remark}\label{remark:limit_traces_F(A)}
Given a $C^*$-algebra $A$ as in the previous proposition, the Cauchy--Schwarz inequality implies that for every approximate unit $(e_n)_{n \in \N}$ for $A$ 
\[\lim_{n \rightarrow \infty} \sup_{\tau \in T_\omega(A)}\lvert \tau(e_n x) - \tau(x)\rvert = 0 \text{ for all } x \in A_\omega.\]
This means in particular that all limit traces vanish on $A_\omega \cap A^\perp$, and that $A_\omega \cap A^\perp \subset J_A$.
\end{remark}

\begin{definition}[cf.\ {\cite{Ozawa}}]\label{definition:uniform_tracial_closure}
Given an algebraically simple $C^*$-algebra $A$ with $T(A)$ non-empty and compact, its \emph{uniform tracial completion} is defined as
\[\bar{A}^u= \frac{\{(a_n)_{n \in \N} \in \ell^\infty(A)\colon (a_n)_{n \in \N} \text{ is } \|\cdot\|_{2,T(A)}\text{-Cauchy} \}}{\{(a_n)_{n \in \N} \in \ell^\infty(A)\colon \lim_{n \rightarrow \infty}\|a_n\|_{2,T(A)} = 0\}}.\]
In other words, it is the $C^*$-algebra obtained by completing the closed unit ball of $A$ with respect to the $\|\cdot\|_{2,T(A)}$-norm and then taking all scalar multiples. If $A$ is separable, Proposition \ref{prop:tracial_ultrapower_unital} guarantees that $\bar{A}^u$ is unital, since an approximate unit $(e_n)_{n \in \N}$ for $A$ will be a $\|\cdot\|_{2,T(A)}$-Cauchy sequence. All traces on $A$ extend to traces on $\bar{A}^u$. As a consequence, $\|\cdot\|_{2,T(A)}$ still defines a valid norm on $\bar{A}^u$.\footnote{It is natural to view the object $\bar{A}^u$ in the category of tracially complete $C^*$-algebras, which will be introduced and clarified in the upcoming paper \cite{CCEGSTW} (currently still in preparation).}
The ultrapower of $\bar{A}^u$ is defined as
\[\left(\bar{A}^u\right)^\omega=  \ell^\infty(\bar{A}^u)/{\{(a_n)_{n \in \N} \in \ell^\infty(\bar{A}^u)\colon \lim_{n \rightarrow \omega}\|a_n\|_{2,T(A)} = 0\}}.\] 
As the unit ball of $A$ is $\|\cdot\|_{2,T(A)}$-dense in the unit ball of $\bar{A}^u$, we get $\left(\bar{A}^u\right)^\omega \cong A^\omega$. 
\end{definition}

Let $A$ be a $C^*$-algebra with tracial state $\tau$. Suppose that $\alpha$ is an automorphism on $A$ that satisfies $\tau \circ \alpha = \tau$. Then $\alpha$ extends uniquely to a trace-preserving automorphism of $\pi_\tau(A)''$, where $\pi_\tau$ denotes the GNS representation of $A$ associated to $\tau$.
\begin{definition}\label{def:strongly_outer}
Let $A$ be an algebraically simple $C^*$-algebra with $T(A) \neq \emptyset$. We say that an automorphism $\alpha$ on $A$ is \emph{strongly outer} if for every $\tau \in T(A)$ with $\tau \circ \alpha = \tau$, the automorphism induced by $\alpha$ on $\pi_\tau(A)''$ is outer. An action $\alpha\colon G \acts A$ of a countable discrete group $G$ is called \emph{strongly outer} if $\alpha_g$ is strongly outer for all $g \in G \setminus\{e\}$.\footnote{In particular, although each automorphism $\alpha$ corresponds to an action of $\Z$, the notion of strongly outer for an action of $\Z$ is much stronger than the corresponding notion for a single automorphism, since it asks that $\alpha^n$ is strongly outer for each $n \neq 0$ instead of just asking for $\alpha$ to be strongly outer.}
\end{definition}
\section{Reducing equivariant $\mathcal{Z}$-stability to a tracial property}\label{section:reduct_1}
In this section we show that --- under the right assumptions --- equivariant $\mathcal{Z}$-stability of an action on a $C^*$-algebra is equivalent to a property at the level of the uniform tracial central sequence algebra of that $C^*$-algebra. This result allows us to reduce the equivariant $\mathcal{Z}$-stability problem to a more manageable tracial one that will be dealt with in later sections. More precisely, we prove the following equivalence:

\begin{theorem}\label{theorem:reduction_tracial_property}
Let $A$ be an algebraically simple, separable, nuclear, $\mathcal{Z}$-stable $C^*$-algebra with $T(A)$ non-empty and compact and let $\alpha\colon G \acts A$ be an action of a countable discrete amenable group $G$. Then the following are equivalent: 
\begin{enumerate}
\item $\alpha \simeq_{\mathrm{cc}}\alpha \otimes \op{id}_\mathcal{Z}$;
\item There exists a unital $*$-homomorphism $M_n \rightarrow (A^\omega \cap A')^{\alpha^\omega}$ for all $n \geq 2$.
\end{enumerate} 
\end{theorem}
This result is folklore and implicit from existing papers, but has never been explicitly stated in this level of generality. Its proof is included here for the sake of completeness and will occupy the rest of this section. It essentially boils down to an equivariant version of a well-known argument due to Matui--Sato \cite{MatuiSatoStrictComparison}. The original argument relies on \emph{property (SI)}, a property for $C^*$-algebras introduced and studied by Matui--Sato in \cite{Sato, MatuiSatoZstability, MatuiSatoStrictComparison, MatuiSatoZstabilityII}, which allows one to derive ordinary $\mathcal{Z}$-stability of a separable $C^*$-algebra $A$ from the existence of unital $*$-homomorphisms $M_n \rightarrow A^\omega\cap A'$. In the case of unital $C^*$-algebras, an equivariant version of property (SI) already appeared implicitly in \cite{Sato} for automorphisms and later in \cite{MatuiSatoZstability} for more general group actions. This was extended in a rather ad hoc way to include certain non-unital $C^*$-algebras by Nawata \cite{Nawata}. The first fully general extension to all simple, nuclear $C^*$-algebras without extra assumptions on the trace space appeared in the work of Szab\'o \cite{EquivariantSI}. There it was also shown that a large class of $C^*$-algebras including the ones in Theorem \ref{theorem:reduction_tracial_property} automatically have property (SI) relative to actions of amenable groups \cite[Corollary 4.3]{EquivariantSI}.

\begin{definition}[special case of {\cite[Definition 2.7]{EquivariantSI}}, see subsequent remark]\label{definition:equivariantSI}
Let $A$ be an algebraically simple, separable $C^*$-algebra with $T(A)$ non-empty and compact, and let $\alpha\colon G \acts A$ be an action by a countable discrete group. We say that $A$ has \emph{equivariant property (SI)} relative to $\alpha$ if the following holds: 

Let $x \in F_\omega(A)^{\tilde{\alpha}_\omega}$  be a positive contraction such that 
\begin{equation}\label{eq:cond_x1}
\tau(ax) =0 \quad \text{for all } a \in A_+\setminus\{0\},\, \tau \in T_\omega(A),
\end{equation}
 and let $y \in F_\omega(A)^{\tilde{\alpha}_\omega}$  be a positive contraction such that for all non-zero $b \in A_+$ there exists a constant $\kappa = \kappa(y,b) >0$ such that \begin{equation}\label{eq:cond_y1}
 \inf_{k \in \N} \tau(by^k) \geq \kappa \tau(b)\quad \text{for all } \tau \in T_\omega(A).
\end{equation}  
Then there exists $s \in F_\omega(A)^{\tilde{\alpha}_\omega}$ such that $s^*s=x$ and $ys=s$. 
\end{definition}

\begin{remark}\label{remark:reduction_definition_SI}
The definition as it is stated in \cite{EquivariantSI} is valid for more general non-unital $C^*$-algebras. It imposes a condition on the elements $x$ and $y$ that is formally stronger than the one in the definition above. However, in case $A$ is  algebraically simple with $T(A)$ compact,  the trace space $T(A)$ is a so-called \emph{compact generator} of the cone of lower semi-continuous tracial weights on $A$.\footnote{This means that $T(A)$ is compact and $\R^{>0} \cdot T(A)$ is equal to the entire cone of such tracial weights except for the trivial zero weight and the weight taking value $\infty$ on all non-zero positive elements.} In this case the original definition can be reduced to the definition given above by \cite[Lemma 2.10 and the subsequent Remark 2.11]{EquivariantSI}. Moreover, it follows from \cite[Proposition 2.4]{EquivariantSI} and simplicity of $A$ that in order to verify condition \eqref{eq:cond_x1} or \eqref{eq:cond_y1} it is enough to check if it holds for an arbitrary single element $a$ or $b$ in $A_+ \setminus\{0\}$ instead of all positive elements. 

Note that in case $G=\{1\}$ or when $\alpha$ is the trivial action, the definition of equivariant property (SI) reduces to the original definition of property (SI) for the $C^*$-algebra.
\end{remark}
We show that under the assumption that the $C^*$-algebra is algebraically simple with compact trace space, Definition \ref{definition:equivariantSI} is equivalent to a statement that agrees with the original definition for actions on unital $C^*$-algebras first appearing in the statement of \cite[Proposition 4.5]{MatuiSatoZstabilityII}: 
\begin{prop}\label{prop:equivalent_statement_SI}
Let $A$ be an algebraically simple, separable $C^*$-algebra with $T(A)$ non-empty and compact, and let $\alpha\colon G \acts A$ be an action by a countable discrete group. Then $A$ has equivariant property (SI) relative to $\alpha$ if and only if the following holds: Let $x,y \in F_\omega(A)^{\tilde{\alpha}^\omega}$ be two positive contractions such that 
\begin{equation}\label{eq:cond_xy}
\tau(x) =0 \quad \text{for all } \tau \in  T_\omega(A) \quad  \text{and} \quad \inf_{k \in \N}\inf_{ \tau  \in T_\omega(A)} \tau(y^k) >0.
\end{equation}
Then there exists an $s \in F_\omega(A)^{\tilde{\alpha}_\omega}$ such that $s^*s=x$ and $ys=s$. 
\end{prop}
\begin{proof}
To prove the equivalence, it is certainly enough to prove that two positive contractions $x,y \in F_\omega(A)^{\tilde{\alpha}_\omega}$ satisfy the conditions \eqref{eq:cond_x1} and \eqref{eq:cond_y1} if and only if they satisfy \eqref{eq:cond_xy}. 
So, take positive contractions $x,y \in F_\omega(A)^{\tilde{\alpha}_\omega}$ satisfying \eqref{eq:cond_xy}. 
Pick an arbitrary non-zero positive $a \in A$ and a $\tau \in T_\omega(A)$. As $x$ commutes with $a$ we get that \[\tau(ax) \leq \tau(x)\|a\|=0,\] thus $x$ satisfies condition \eqref{eq:cond_x1}. 
By Remark \ref{remark:limit_traces_F(A)} it is possible to pick a positive contraction $e \in A$ such that 
\[\sup_{\tau \in T_\omega(A)}|\tau(y) - \tau(ey)| < \inf_{k \in \N}\inf_{ \tau  \in T_\omega(A)} \tau(y^k).\] As explained in Remark \ref{remark:reduction_definition_SI} it suffices to verify  condition \eqref{eq:cond_y1} for an arbitrary single element $b \in A_+\setminus\{0\}$, so in particular we can pick $b=e$. Set 
\[\kappa := \inf_{k \in \N}\inf_{ \tau  \in T_\omega(A)} \tau(y^k)- \sup_{\tau \in T_\omega(A)} |\tau(y) - \tau(ey)| >0.\]
Then for any $k \in \N$ and $\tau \in T_\omega(A)$ we have
\begin{align*}
\tau(ey^k)&= \tau(y^k) - \tau(y^k) + \tau(ey^k)\\
 &\geq \tau(y^k)- \left(\tau(y) - \tau(ey)\right)\,\|y^{k-1}\|\\
 &\geq \kappa \geq \kappa \tau(e).
\end{align*} 
This shows that $y$ satisfies \eqref{eq:cond_y1}.

For the other implication, assume $x,y \in F_\omega(A)^{\tilde{\alpha}_\omega}$ are positive contractions satisfying \eqref{eq:cond_x1} and \eqref{eq:cond_y1}. 
Take an approximate unit $(e_n)_{n \in \N}$ for $A$.
By Remark \ref{remark:limit_traces_F(A)} we get ${\tau(x) = \lim_{n \rightarrow \infty}\tau(e_nx) = 0}$ for all $\tau \in T_\omega(A)$. 
By Proposition \ref{prop:tracial_ultrapower_unital} we can pick $n \in \N$ large enough such that $\tau(e_n) >1/2$ for all $\tau \in T_\omega(A)$. 
Hence
\[\tau(y^k) \geq \tau(y^{k/2}e_ny^{k/2}) = \tau(e_ny^k) \geq \kappa(y,e_n)\tau(e_n) > \frac{1}{2}\kappa(y,e_n)>0\]
for all $k \in \N$ and $\tau \in T_\omega(A)$. We conclude that $x$ and $y$ satisfy \eqref{eq:cond_xy}. 
\end{proof}
A second ingredient in the proof of Theorem \ref{theorem:reduction_tracial_property} is a generalization of Kirchberg's concept of a $\sigma$-ideal \cite[Definition 1.5]{Kirchberg} to the equivariant context. For $\Z$-actions, this was already introduced by Liao in \cite[5.5]{Liao}. The definition was generalized to actions of more general countable discrete groups in {\cite[Definition 4.1]{SsaDynSyst2}}. It reduces to Kirchberg's original definition when $G = \{1\}$. 
 \begin{definition}[{\cite[Definition 4.1]{SsaDynSyst2}}]
Let $A$ be a $C^*$-algebra with an action $\alpha\colon G \acts A$ of a countable discrete group. An $\alpha$-invariant ideal $J \subset A$ is called an \emph{$G$-$\sigma$-ideal} if for every separable, $\alpha$-invariant $C^*$-subalgebra $C \subset A$ there exists a positive contraction $e \in (J \cap C')^\alpha$ such that $ec=c$ for all $c \in J \cap C$. 
\end{definition}

These kind of ideals are mainly useful because of the following interesting property they have:
\begin{prop}[{\cite[Proposition 4.5]{SsaDynSyst2}}]\label{prop:sigma_ideal}
Let $A$ be a $C^*$-algebra with an action $\alpha\colon G \acts A$ of a countable discrete group $G$, and let $J \subset A$ be a $G$-$\sigma$-ideal. Let $B=A/J$ and $\beta\colon G \acts B$ the action induced by $\alpha$. Denote by $\pi\colon A \rightarrow B$ the equivariant quotient map. Then
\begin{enumerate}
\item For every $\alpha$-invariant, separable $C^*$-subalgebra $D \subset A$, the induced map $\pi\colon A \cap D' \rightarrow B \cap \pi(D)'$ is surjective and its kernel $J \cap D'$ is a $G$-$\sigma$-ideal in $A \cap D'$.
\item For every $\beta$-invariant, separable $C^*$-algebra $C \subset B$, there exists an equivariant c.p.c.\ order zero map $\psi\colon (C, \beta) \rightarrow (A, \alpha)$ such that $\pi \circ \psi = \op{id}_C$. 
\end{enumerate}  
\end{prop}

\begin{example}\label{example:sigma_ideal}
A well-known example occurs when $A$ is a $C^*$-algebra with $T(A)$ non-empty and compact, and $\alpha\colon G \acts A$ is an action of a countable discrete group. 
In this case the trace kernel ideal $J_A$ is a $G$-$\sigma$-ideal of $A_\omega$. 
Indeed, let $C \subset A_\omega$ be a separable, $\alpha_\omega$-invariant $C^*$-subalgebra. 
Then the ideal $J_A \cap C$ admits a countable approximate unit that is approximately $\alpha_\omega$-invariant and quasi-central relative to $C$ by \cite[Lemma 1.4]{Kasparov} and hence, there is a sequence $(e_n)_{n \in \N} \in J_A$ such that 
\begin{align*}
\|e_n - \alpha_{\omega,g}(e_n)\| &\rightarrow 0 \text{ for all }g \in G,\\
\|e_nc-c\| &\rightarrow 0 \text { for all } c \in J_A \cap C \text{, and } \\
\|[e_n,c]\| &\rightarrow 0 \text{ for all }c \in C.\end{align*} Applying Kirchberg's $\e$-test \cite[Lemma A.1]{Kirchberg} to a countable set of conditions corresponding to a countable dense subset of $C$ then yields the required element $e$. 

If $A$ is separable, it follows from applying Proposition \ref{prop:sigma_ideal} to the $G$-$\sigma$-ideal $J_A \subset A_\omega$ that the canonical map 
\[{A_\omega \cap A' \rightarrow A^\omega \cap A'}\] is surjective, and that its kernel $J_A \cap A'$ is again a $G$-$\sigma$-ideal of $A_\omega \cap A'$. Recall from Remark \ref{remark:limit_traces_F(A)} that the ideal $A_\omega \cap A^\perp \subset A_\omega\cap A'$ is also a subset of $J_A \cap A'$. An easy argument using the definition of a $G$-$\sigma$-ideal (see \cite[Proposition 5.17]{EquivariantSI} for a general proof) shows that  
\[\mathcal{J}_A := (J_A \cap A')/(A_\omega \cap A^\perp)\] is a $G$-$\sigma$-ideal in $(A_\omega \cap A')/(A_\omega \cap A^\perp) = F_\omega(A)$. Thus, the kernel of the natural quotient map \[F_\omega(A) \rightarrow A^\omega \cap A'\] is also $G$-$\sigma$-ideal.
\end{example}

\begin{proof}[{Proof of Theorem \ref{theorem:reduction_tracial_property}}]
We first prove the implication $(1) \Rightarrow (2)$. 
Assume that $\alpha$ and $\alpha \otimes \op{id}_\mathcal{Z}$ are cocycle conjugate. 
Since $A$ is separable, this is equivalent to the existence of a unital $*$-homomorphism $\mathcal{Z} \rightarrow F_\omega(A)^{\tilde{\alpha}_\omega}$, see \cite[Corollary 3.8]{SsaDynSyst}. 
Composing with the canonical quotient map gives a unital $*$-homomorphism $\phi\colon \mathcal{Z} \rightarrow (A^\omega \cap A')^{\alpha^\omega}$. The unit ball of $A^\omega$ is complete for the topology induced by the $\|\cdot\|_{2,T_\omega(A)}$-norm (e.g.\ \cite[Lemma 1.6]{CETWW}), and hence the same is true for the unit ball of $(A^\omega \cap A')^{\alpha^\omega}$. 
Let $\tau_\mathcal{Z}$ denote the unique trace on $\mathcal{Z}$. Since $\phi$ is $(\|\cdot\|_{2, \tau_{\mathcal{Z}}}, \|\cdot\|_{2,T_{\omega}(A)})$-contractive, by Kaplansky's density theorem it can be extended to a map from $\pi_{\tau_{\mathcal{Z}}}(\mathcal{Z})''$, the weak closure of $\mathcal{Z}$ under the GNS representation associated to $\tau_\mathcal{Z}$. 
For any $n \in \N$, composing with a unital embedding $M_n \hookrightarrow  \pi_{\tau_{\mathcal{Z}}}(\mathcal{Z})'' \cong \mathcal{R}$ yields a unital $*$-homomorphism $M_n \rightarrow (A^\omega \cap A')^{\alpha^\omega}$.

 To prove the other implication, assume that there exists a unital $*$-homomorphism $M_n \rightarrow (A^\omega \cap A')^{\alpha^\omega}$ for some $n \geq 2$. As explained in Example \ref{example:sigma_ideal}, the ideal 
 \[\mathcal{J}_A = (J_A \cap A')/(A_\omega \cap A^\perp) \subset F_\omega(A)\] is a $G$-$\sigma$-ideal and hence it follows from Proposition \ref{prop:sigma_ideal} that there is a c.p.c.\ order zero map $\phi\colon M_n \rightarrow F_\omega(A)^{\tilde{\alpha}_\omega}$ lifting this map. 
 In particular $1-\phi(1) \in \mathcal{J}_A$. Let $p \in M_n$ denote a rank one projection. Then $\phi(p) - \phi(p)^m \in \mathcal{J}_A$ for all $m \in \N$, so $\phi(p)$ agrees with all its powers on limit traces. By uniqueness of the tracial state on $M_n$, we get
\[\inf_{m \in \N}\inf_{\tau \in T_\omega(A)} \tau(\phi(p)^m)=1/n >0.\]
The assumptions on $A$ guarantee it has equivariant property (SI) relative to all actions of countable discrete amenable groups, see \cite[Theorem B]{EquivariantSI}.  Using the equivalent formulation of equivariant property (SI) from Proposition \ref{prop:equivalent_statement_SI}, we can find an ${s \in F_\omega(A)^{\tilde{\alpha}_\omega}}$ such that $s^*s = 1- \phi(1)$ and $\phi(p)s = s$. By \cite[Proposition 5.1]{RordamWinter} this implies the existence of a unital $*$-homomorphism from the prime dimension drop algebra $Z_{n, n+1} \rightarrow F_\omega(A)^{\tilde{\alpha}_\omega}$. Since we can do this for all $n \geq 2$ and $\mathcal{Z}$ is an inductive limit of such dimension drop algebras, it follows by a standard argument (analogous to the proof of \cite[Proposition 2.2]{TomsWinter08}) that there exists a unital $*$-homomorphism $\mathcal{Z} \rightarrow F_\omega(A)^{\tilde{\alpha}_\omega}$. As already stated, this is equivalent to equivariant $\mathcal{Z}$-stability of $\alpha$. 
\end{proof}

\section{Rokhlin dimension with commuting towers for strongly outer actions}\label{section:reduct_2}
In this section we show that in order to prove the second main result of this paper (Theorem \ref{theorem:intro_main2}) about the Rokhlin dimension with commuting towers of certain strongly outer\footnote{Recall that we call an action $\alpha\colon G \acts A$ \emph{strongly outer} if for every $g \in G\setminus\{e\}$ and every $\tau \in T(A)^{\alpha_g}$, the weak extension of $\alpha_g$ in the GNS representation of $A$ associated to $\tau$ is outer.} $\Z$-actions, it suffices to prove a more manageable tracial result. 
 The notion of Rokhlin dimension with commuting towers was originally introduced by Hirshberg--Winter--Zacharias in \cite{HirshbergWinterZacharias} for actions of the integers and finite groups. It was generalized to actions of residually finite groups by Szab\'o--Wu--Zacharias in \cite{SzaboWuZacharias}:

\begin{definition}{\cite[Definition 10.2]{SzaboWuZacharias}}
Let $\alpha\colon G \acts A$ be an action of a residually finite group $G$ on a separable $C^*$-algebra.
\begin{enumerate}
\item Let $H \subset G$ be a subgroup of finite index. The Rokhlin dimension of $\alpha$ with commuting towers relative to $H$, denoted $\op{dim}^c_{\op{Rok}}(\alpha,H)$, is the smallest natural number $d$ such that there exist equivariant c.p.c.\ order zero maps 
\[\phi^{(0)}, \hdots, \phi^{(d)}\colon (C(G/H),G\text{-shift}) \rightarrow (F_\omega(A), \tilde{\alpha}_\omega)\]
with pairwise commuting ranges such that
\[\phi^{(0)}(1) + \hdots + \phi^{(d)}(1) = 1.\]
\item The Rokhlin dimension of $\alpha$ with commuting towers is defined as 
\[\dim^c_{\op{Rok}}(\alpha) := \sup\{\op{dim}^c_{\op{Rok}}(\alpha,H) \mid H \subseteq G \text{ of finite index}\}\]
\end{enumerate}
\end{definition}

 The following result, which is a direct consequence of \cite[Corollary 5.1]{RokhlinDimension} (using \cite[Theorem 4.10]{SzaboWuZacharias} to verify one additional necessary condition), illustrates the importance of having finite Rokhlin dimension:
\begin{theorem}\label{theorem:finite_Rokhlin_dimension}
Let $\alpha\colon G \acts A$ be an action of a finitely generated, virtually nilpotent group $G$ on a separable $C^*$-algebra $A$. Let $\mathcal{D}$ be a strongly self-absorbing $C^*$-algebra with $A \cong A \otimes \mathcal{D}$. If $\op{dim}^c_{\op{Rok}}(\alpha) < \infty$, then $\alpha \simeq_{\mathrm{cc}} \alpha \otimes \delta$ for all strongly self-absorbing actions $\delta\colon G \acts \mathcal{D}$.
\end{theorem}

The next theorem shows exactly how the problem of proving that a certain $\Z$-action has finite Rokhlin dimension with commuting towers can be reduced to a tracial problem:

\begin{theorem}\label{theorem:reduction_rokhlin_dimension}
Let $A$ be an algebraically simple, separable, nuclear, $\mathcal{Z}$-stable $C^*$-algebra with $T(A)$ non-empty and compact. Fix an action $\alpha\colon \Z \acts A$. Assume that for each unitary representation $\nu\colon \Z \rightarrow \mathcal{U}(M_n)$ there exists a unital equivariant $*$-homomorphism $(M_n, \op{Ad}(\nu)) \rightarrow (A^\omega \cap A', \alpha^\omega)$. Then $\op{dim}^c_{\op{Rok}}(\alpha) \leq 2$. 
\end{theorem}
\begin{proof}
This result follows from a similar argument as in the proof of \cite[Theorem 2.14]{ActionsTorsionFree}, but since the context there differs a bit from ours we provide a proof here for clarity. By Kirchberg's $\e$-test (\cite[Lemma A.1]{Kirchberg}) it suffices to show that for any fixed $\e>0$ and $n \geq 2$, there exist pairwise commuting positive contractions $a_j,b_j,c_j \in F_\omega(A)$ for $j \in 0, \hdots, n-1$ such that
\begin{enumerate}[itemsep=1ex,topsep=1ex]
\item\label{firstitem} $1=\sum_{j=0}^{n-1} (a_j + b_j + c_j)$;
\item the collections $\{a_j\}_{j=0}^{n-1}, \{b_j\}_{j=0}^{n-1}$ and $\{c_j\}_{j=0}^{n-1}$ each consist of pairwise orthogonal elements;
\item $\|\tilde{\alpha}_\omega(a_j) - a_{j+1 \bmod n}\| < \e;$
\item $\|\tilde{\alpha}_\omega(b_j) - b_{j+1 \bmod n}\| < \e;$ and
\item\label{lastitem} $\|\tilde{\alpha}_\omega(c_j) - c_{j+1 \bmod n}\| < \e$.
\end{enumerate}  
It follows from \cite[Proposition 5.1]{RordamWinter} that for every $N \geq 2$, the prime dimension drop algebra $Z_{N,N+1}$ is isomorphic to the universal unital $C^*$-algebra $Z^U_{N,N+1}$ generated by a contraction $v \in Z^U_{N,N+1}$ and the image of a c.p.c.\ order zero map $\psi\colon M_N \rightarrow Z^U_{N,N+1}$ satisfying $1- \psi(1) = v^*v$ and $\psi(e_{1,1})v=v$, where $e_{1,1}$ denotes the rank one projection in $M_N$ with a one in the upper left corner. When $\nu\colon \Z \rightarrow \mathcal{U}(M_{N-1})$ is a unitary representation, we denote by $\delta^\nu\colon \Z \acts Z_{N,N+1}^U$ the action determined by $\delta^\nu_z(v)=v$ and $\delta_z^\nu \circ \psi = \psi \circ \op{Ad}(1 \oplus \nu_z)$ for all $z \in \Z$.
 Fix $\e>0$ and $n \geq 2$. It was shown in \cite[Section 2]{ActionsTorsionFree} that there exists an $N \geq 2$ dependent on $\e$ and $n$ and a unitary representation $\nu\colon \Z \rightarrow \mathcal{U}(M_{N-1})$ such that the $C^*$-algebra $Z^U_{N,N+1}$ contains elements with the necessary  properties \ref{firstitem}-\ref{lastitem} above (replacing the action $\tilde{\alpha}_\omega$ by $\delta^\nu$). As a consequence, it suffices to prove that given $N \geq 2$ and $\nu\colon \Z \rightarrow \mathcal{U}(M_{N-1})$, there exists a unital equivariant $*$-homomorphism from $(Z^U_{N,N+1}, \delta^\nu)$ to $(F_\omega(A), \tilde{\alpha}_\omega)$. This is precisely what we will do.
 
By assumption there exists a unital equivariant $*$-homomorphism 
\[(M_N, \op{Ad}(1 \oplus \nu)) \rightarrow (A^\omega \cap A',\alpha^\omega).\] Recall from Example \ref{example:sigma_ideal} that the ideal 
 \[\mathcal{J}_A = (J_A \cap A')/(A_\omega \cap A^\perp) \subset F_\omega(A)\] is a $\Z$-$\sigma$-ideal. Therefore, it follows from Proposition \ref{prop:sigma_ideal} that there is an equivariant c.p.c.\ order zero map $\phi\colon (M_N, \op{Ad}(1 \oplus \nu)) \rightarrow (F_\omega(A),{\tilde{\alpha}_\omega})$ lifting this map. In particular it holds that $1-\phi(1) \in \mathcal{J}_A$. Since $e_{1,1} \in M_N$ is invariant under $\op{Ad}(1\oplus \nu)$, we get $\phi(e_{1,1}) \in F_\omega(A)^{\tilde{\alpha}_\omega}.$ Moreover, since $e_{1,1}$ is a projection it also holds that $\phi(e_{1,1}) - \phi(e_{1,1})^m \in \mathcal{J}_A$ for all $m \in \N$, so $\phi(e_{1,1})$ agrees with all its powers on limit traces. By uniqueness of the tracial state on $M_N$, we get
\[\inf_{m \in \N}\inf_{\tau \in T_\omega(A)} \tau(\phi(e_{1,1})^m)=1/N >0.\]
The assumptions on $A$ imply that it has equivariant property (SI) relative to all actions of countable amenable groups, see \cite[Corollary 4.3]{EquivariantSI}. By the equivalent formulation of property (SI) from Proposition \ref{prop:equivalent_statement_SI}, this means we can find a positive contraction ${s \in F_\omega(A)^{\tilde{\alpha}_\omega}}$ such that $s^*s = 1- \phi(1)$ and $\phi(e_{1,1})s = s$. By definition of the universal $C^*$-algebra $Z_{N,N+1}^U$ this implies the existence of a unique unital $*$-homomorphism $\chi\colon Z^U_{N, N+1} \rightarrow F_\omega(A)$ such that $\chi(v) = s$ and $\chi \circ \psi =\phi$. In particular $\chi$ must also be equivariant with respect to $\delta^\nu$ and $\tilde{\alpha}_\omega$. This concludes the proof.
\end{proof}
\section{Systems generated by ranges of pairwise commuting order zero maps}\label{section:syst_order_zero}
In the previous two sections we demonstrated that in order to prove the main results of this paper about automorphisms $\alpha  \in \op{Aut}(A)$ or the corresponding actions $\alpha\colon \Z \acts A$, it suffices to find unital equivariant $*$-homomorphisms $(M_n, \op{Ad}(v)) \rightarrow (A^\omega \cap A',{\alpha^\omega})$ for all kinds of unitaries $v \in M_n$.  This problem can still be reduced a bit further: it suffices to find equivariant pairwise commuting c.p.c.\ order zero maps $\phi_i\colon (M_n, \op{Ad}(v)) \rightarrow (A^\omega \cap A',{\alpha^\omega})$ for $i=1, \hdots, m$ such that $\sum_{i=1}^m \phi_i(1) = 1$. This result is probably well known to experts, but does not appear explicitly in the literature in this form. We will prove it here for the sake of completeness. In the next section we will build on this and prove a slightly adapted, approximate version better suited to our needs (Lemma \ref{lemma:join}).

The following construction and observations are taken from \cite[Section 3]{RokhlinDimension} 
and have their origins in \cite[Section 5]{HirshbergWinterZacharias}:
\begin{notation}
Let $D_1, \hdots, D_m$ be unital $C^*$-algebras. For $t \in [0,1]$ and $j=1, \hdots, m$ we denote
\[D_j^{(t)} := \begin{cases}
D_j & \text{if } t >0,\\
\C \cdot 1_{D_j} & \text{if } t=0.
\end{cases}\]
For a tuple $\vec{t} = (t_1, \hdots, t_m) \in [0,1]^m$, we define
\[D^{(\vec{t})} := D_1^{(t_1)} \otimes_\textup{max} \hdots \otimes_\textup{max} D_m^{(t_m)}.\]
Consider the simplex 
\[\Delta^{(m)}:= \{\vec{t} \in [0,1]^m \mid t_1 + \hdots + t_m = 1\}\]
and define
\[\mathcal{E}(D_1, \hdots, D_m) := \{f \in C(\Delta^{(m)}, D_1 \otimes_{\textup{max}} \hdots \otimes_{\textup{max}} D_m) \mid f(\vec{t}) \in D^{(\vec{t})}\}.\]
For every $j = 1,  \hdots, m$ there is a canonical c.p.c.\ order zero map 
\[\eta_j \colon D_j \rightarrow \mathcal{E}(D_1, \hdots, D_m)\]
given by
\[\eta_j(d_j)(\vec{t}) = t_j \cdot (1_{D_1} \otimes \hdots \otimes 1_{D_{j-1}} \otimes d_j \otimes 1_{D_{j+1}} \otimes \hdots \otimes 1_{D_m}).\]
It is easy to check that the ranges of these $\eta_j$ generate $\mathcal{E}(D_1, \hdots D_m)$ as a $C^*$-algebra.

Let $G$ be a countable discrete group. Given actions $\delta^{(1)}\colon G \acts D_1, \hdots, \delta^{(m)}\colon G \acts D_m$, the action on $C(\Delta^{(m)}, D_1 \otimes_{\max} \hdots \otimes_{\max} D_m)$ that is defined fiberwise by $\delta^{(1)} \otimes \hdots \otimes \delta^{(m)}$ restricts to a well-defined action 
 \[\mathcal{E} \left(\delta^{(1)}, \hdots, \delta^{(m)}\right) \colon G \acts \mathcal{E}(D_1, \hdots , D_m).\]
\end{notation}

\begin{prop}\label{prop:universal_property}
Let $D_1, \hdots, D_m$ be $m$ unital $C^*$-algebras with actions ${\delta^{(1)}\colon G \acts D_1,\hdots,}$ ${\delta^{(m)}\colon G \acts D_m}$ of a countable discrete group $G$. Then $\mathcal{E}(D_1, \hdots D_m)$ and $\mathcal{E}\left(\delta^{(1)}, \hdots \delta^{(m)}\right)$ together with the c.p.c.\ order zero maps $\eta_j$ satisfy the following universal property:
If $B$ is a unital $C^*$-algebra with action $\beta\colon G \acts B$ and $\psi_j\colon \left(D_j,\delta^{(j)}\right) \rightarrow (B,\beta)$ for $j=1,\hdots, m$ are equivariant c.p.c.\ order zero maps with pairwise commuting ranges such that
\[\psi_1(1_{D_1}) + \hdots +  \psi_m(1_{D_m}) = 1_B,\]
then there exists a unique unital equivariant $*$-homomorphism \[\Psi\colon\left(\mathcal{E}(D_1, \hdots D_m), \mathcal{E}\left(\delta^{(1)}, \hdots, \delta^{(m)}\right)\right) \rightarrow (B,\beta)\] such that $\Psi \circ \eta_j = \psi_j$ for $j=1,\hdots, m$. 
\end{prop}

There is another natural way to view the $C^*$-algebra $\mathcal{E}(D_1, \hdots, D_m)$:
\begin{notation} 
Given two unital $C^*$-algebras $D_1$ and $D_2$, one can define a new $C^*$-algebra by
\[D_1 \star D_2 :=\{f \in C([0,1], D_1 \otimes_{\max} D_2 \mid f(0) \in D_1 \otimes 1, f(1) \in 1 \otimes D_2\}.\] 
We call this the \emph{join} of the two $C^*$-algebras. As shown in \cite[Remark 3.5]{RokhlinDimension}, there is a natural isomorphism
\[\mathcal{E}(D_1,D_2) \cong D_1 \star D_2.\]
More generally, it is shown there that for each $m \geq 2$ there is an isomorphism
\[\mathcal{E}(D_1, \hdots, D_{m}) \cong D_1 \star \mathcal{E}(D_2, \hdots, D_{m})\]
that becomes equivariant in the presence of actions $\delta^{(1)}\colon G \acts D_1, \hdots, \delta^{(m)}\colon G \acts D_m$. Thus, it makes sense to view the $C^*$-algebra $\mathcal{E}(D_1, \hdots, D_{m})$ as the $m$-fold join $D_1 \star \hdots \star D_m$. From now on, we will denote
$D^{\star m} := \mathcal{E}(\underbrace{D, \hdots, D}_\text{$m$ times})$ and $\delta^{\star m} = \mathcal{E}(\underbrace{\delta, \hdots, \delta}_\text{$m$ times})$. 
\end{notation}

In the case where all $D_j$ are the same matrix algebra $M_n$, the universal property leads to the following result:
\begin{lemma}\label{lemma:join_basic}
Let $\nu\colon G \rightarrow \mathcal{U}(M_n)$ be a unitary representation of a countable discrete group. For any unital $C^*$-algebra $B$ equipped with an action $\beta\colon G \acts B$ the following holds: Suppose $\psi_j\colon (M_n, \op{Ad}(\nu)) \rightarrow (B,\beta)$ for $j=1, \hdots, m$ are equivariant c.p.c.\ order zero maps with pairwise commuting ranges such that ${\sum_{j=1}^m \psi_j(1) = 1}$. Then there exists a unital equivariant $*$-homomorphism ${\Psi\colon(M_n,\op{Ad}(\nu)) \rightarrow (B,\beta)}$.
\end{lemma}
\begin{proof}
Suppose there exist c.p.c.\ order zero maps $\psi_j$ for $j=1, \hdots, m$ that satisfy all the properties in the statement of the lemma. By the universal property (Proposition \ref{prop:universal_property}) we get a unital equivariant $*$-homomorphism $(M_n^{\star m},\op{Ad}(\nu)^{\star m}) \rightarrow (B,\beta)$. Hence, to prove the lemma it suffices to prove that there is a unital equivariant embedding of $(M_n,\op{Ad}(\nu))$ into $(M_n^{\star m}, \op{Ad}(\nu)^{\star m})$. 

We start by proving there is a unital equivariant embedding \[(M_n,\op{Ad}(\nu)) \rightarrow (M_n \star M_n, \op{Ad}(\nu)^{\star 2}).\]
Recall that the flip automorphism on $M_n \otimes M_n$ is implemented by conjugation by the unitary $\sum_{i,j=1}^n e_{ij} \otimes e_{ji}$ (where the $e_{ij}$ denote the standard matrix units). By functional calculus there exists a self-adjoint element $a \in M_n \otimes M_n$ such that 
$\sum_{i,j=1}^n e_{ij} \otimes e_{ji}= \exp(ia).$ Then the map
 \[\iota\colon M_n \rightarrow M_n \star M_n\colon x \mapsto [t \mapsto \exp(ita)(x \otimes 1) \exp(ita)^*]\] 
 is a well-defined unital $*$-homomorphism. It is also equivariant: Take $v \in \mathcal{U}(M_n)$ arbitrarily. Since we performed functional calculus in the commutant of the symmetric elementary tensors $a$ commutes with $v\otimes v$ and hence, $\exp(ita)$ also commutes with $v \otimes v$  for all $t \in [0,1]$. So, we get
 \begin{align*}
 \iota(vxv^*)(t) &= \exp(ita)(vxv^* \otimes 1) \exp(ita)^*\\
 &= \exp(ita) (v \otimes v) (x \otimes 1) (v^* \otimes v^*)\exp(ita)^*\\
 &=  (v \otimes v)\exp(ita) (x \otimes 1) \exp(ita)^*(v^* \otimes v^*)\\
 &= (v \otimes v) \iota(x)(t)(v^* \otimes v^*).
 \end{align*}
 This proves equivariance.
In general, given a unital equivariant inclusion \[\phi\colon \left(M_n^{\star (k-1)},\op{Ad}(\nu)^{\star (k-1)}\right)\rightarrow \left(M_n^{\star k},\op{Ad}(\nu)^{\star k}\right)\]  for $k \geq 2$, functoriality of the join construction yields a unital equivariant embedding of $M_n \star M_n ^{\star (k-1)}$ into $ M_n \star M_n^{\star k}$ that is
induced by the identity map in the first variable and $\phi$ in the second variable.
 By induction, we get a unital equivariant map \[\left(M_n^{\star l},\op{Ad}(\nu)^{\star l}\right) \rightarrow \left(M_n^{\star(l+1)},\op{Ad}(\nu)^{\star (l+1)}\right)\] for each $l\in \N$, so in particular for $l=1, \hdots, m-1$. The composition of these maps is a unital equivariant embedding of $M_n$ into $M_n^{\star m}$. 
\end{proof}

\section{McDuff properties for automorphisms on $W^*$-bundles}\label{section:W*-bundles}
In Section 2 and 3 we reduced the two main problems considered in this paper to problems about the uniform tracial central sequence algebra of the $C^*$-dynamics involved. In this section we tackle this tracial problem. Given an algebraically simple $C^*$-algebra $A$ with $T(A)$ non-empty and compact, the central sequence algebras $A^\omega \cap A'$ and $(\bar{A}^u)^\omega \cap (\bar{A}^u)'$ are isomorphic. This means that we are actually interested in proving properties about the uniform tracial closure of the $C^*$-algebra. For the class of $C^*$-algebras we consider, this tracial closure has a natural $W^*$-bundle structure that we will  make use of. 

\subsection{Preliminaries about $W^*$-bundles}
\begin{definition}[{\cite[Section 5]{Ozawa}}]
Let $X$ be a compact metrizable space. A $W^*$-bundle over $X$ is a triple $(\mathcal{M}, X, E)$ such that
\begin{enumerate}[topsep=2pt]
\item $\mathcal{M}$ is a unital $C^*$-algebra;
\item there is a given unital inclusion of $C(X)$ into the center of $\mathcal{M}$\footnote{The choice of inclusion is part of the definition, but we omit it throughout in order to keep the notation lighter.};
\item $E\colon \mathcal{M} \rightarrow C(X)$ is a faithful conditional expectation satisfying $E(ab) = E(ba)$ for all $a,b \in \mathcal{M}$;
\item The unit ball $\{a \in \mathcal{M}\colon \|a\| \leq 1\}$ is complete with respect to the \emph{uniform $2$-norm} defined by $\|a\|_{2,u}= \|E(a^*a)^{1/2}\|_{C(X)}$.
\end{enumerate}
We say the bundle is \emph{strictly separable} if it contains a countable subset that is dense in the uniform 2-norm. When clear from context, the base space $X$ and conditional expectation $E$ will often be omitted from the notation.  
\end{definition}

For $x \in X$, we can define a tracial state $\tau_x$ on $\mathcal{M}$ by $\tau_x = \op{eval}_x \circ E$. Let $\pi_x\colon \mathcal{M} \rightarrow \mathcal{B}(L^2(\mathcal{M}, \tau_x))$ denote the associated GNS representation. Write $\mathcal{M}_x := \pi_x(\mathcal{M})$. Then the \emph{fiber} of $\mathcal{M}$ over $x$ is the $C^*$-algebra $\mathcal{M}_x$ with trace $\tau_x$. This will in fact be a tracial von Neumann algebra (meaning that $\tau$ is a faithful normal trace).

\begin{example}[Trivial $W^*$-bundles]\label{example:W-bundles}
~\begin{itemize}
\item
Let $X$ be a compact metrizable space and let $(M, \tau)$ be a tracial von Neumann algebra. Consider the unital $C^*$-algebra
\[C_\sigma(X,M) := \{f\colon X \rightarrow M\colon f \text{ is  }\|\cdot\|\text{-bounded and } \|\cdot\|_{2,\tau}\text{-continuous}.\}\]
Together with the obvious embedding $C(X) \rightarrow \mathcal{Z}(C_\sigma(X,M))$ and the conditional expectation $E\colon C_\sigma(X,M) \rightarrow C(X)$ defined by $E(f)(x) = \tau(f(x))$ this is a $W^*$-bundle. It is called the \emph{trivial $W^*$-bundle} over $X$ with fiber $M$. 
\item In particular, tracial von Neumann algebras are trivial $W^*$-bundles over a one-point base space.
\item One of the motivations of Ozawa for defining $W^*$-bundles were special cases of these trivial $W^*$-bundles:
When $A$ is an algebraically simple, separable, nuclear $C^*$-algebra for which $T(A)$ is a non-empty Bauer simplex with finite-dimensional extreme boundary, he showed that $\bar{A}^u$ has a natural structure of a $W^*$-bundle over $\partial_eT(A)$ \cite[Theorem 4]{Ozawa}. The conditional expectation $E\colon \bar{A}^u \rightarrow C(\partial_eT(A))$ satisfies $E(a)(\tau) = \tau(a)$ for all $a \in A$ and $\tau \in \partial_eT(A)$.\footnote{Note that in this way, the uniform 2-norm it gets as a $W^*$-bundle is equal to $\|\cdot\|_{2,\partial_eT(A)}$. The Krein--Milman theorem guarantees that this equals $\|\cdot\|_{2,T(A)}$.} In this way, the fiber over $\tau \in \partial_e T(A)$ can be canonically identified with $\pi_\tau(A)''$. Since $A$ is nuclear we have $\pi_\tau(A)'' \cong \mathcal{R}$ for each $\tau \in \partial_e T(A)$. These isomorphisms can be coordinated in such a way that they give rise to a $*$-homomorphism $\pi\colon A \rightarrow C_\sigma (\partial_eT(A) , \mathcal{R})$ satisfying $\op{eval}_\tau \circ \pi = \pi_\tau$ for each $\tau \in \partial_e T(A)$, where $\pi_\tau$ denotes the GNS representation induced by $\tau$, and such that the image of $\pi$ is dense in the uniform 2-norm. Thus, $\bar{A}^u$ is isomorphic to the trivial bundle $C_\sigma (\partial_eT(A) , \mathcal{R})$. 
\end{itemize}
\end{example}

\begin{definition}[Morphisms of $W^*$-bundles, cf.\ {\cite[Definition 2.3]{EvingtonPennig}}]\label{def:morphism_W*}
Let $\mathcal{M}_i$ be a $W^*$-bundle over $X_i$ with conditional expectation $E_i$ for $i=1,2$. A \emph{morphism} is a unital $*$-homomorphism $\phi\colon \mathcal{M}_1 \rightarrow \mathcal{M}_2$ such that $\phi(C(X_1)) \subseteq C(X_2)$ and $\phi \circ E_1 = E_2 \circ \phi$.
\end{definition}

\begin{remark}\label{remark:action_W*bundles}
The above notion of a morphism between $W^*$-bundles is stronger than that of a $*$-homomorphism between the corresponding $C^*$-algebras. For the rest of this paper, whenever we mention a morphism or automorphism on  a $W^*$-bundle $(\mathcal{M},X,E)$, we will always mean this in the more strict sense of Definition \ref{def:morphism_W*} above. When we talk about an action $\gamma\colon G \acts \mathcal{M}$, we will always mean that $\gamma$ acts by automorphisms in this sense as well. In particular, this means that $\gamma_g(C(X)) = C(X)$ and $\gamma_g \circ E = E \circ \gamma_g$ for each $g \in G$.
\end{remark}

\begin{definition}[Quotients of $W^*$-bundles, cf.\ {\cite[Definition 2.8 and Proposition 2.9]{EvingtonPennig}}]\label{defintion:quotient_bundles}
Let $(\mathcal{M},X,E)$ be a $W^*$-bundle and take a closed subset $Y \subseteq X$. We define:
\[I_Y = \{a \in \mathcal{M} \mid E(a^*a)(y) = 0 \text{ for all } y \in Y\}.\]
This is a $\|\cdot\|_{2,u}$-closed, two-sided ideal of $\mathcal{M}$ and $E(I_Y) \subseteq I_Y$.

The quotient $\mathcal{M}/I_Y$ can be given the structure of a $W^*$-bundle over $Y$. Let $q$ denote the canonical quotient map, let $\iota\colon C(X) \rightarrow \mathcal{M}$ denote the unital embedding into the center and let $r\colon C(X) \rightarrow C(Y)$ denote the restriction map. Then there exist a unital central embedding $ \iota_Y\colon C(Y) \rightarrow \mathcal{M}/I_Y$ and a conditional expectation $E_Y\colon \mathcal{M}/{I_Y} \rightarrow C(Y)$ such that $q \circ \iota = \iota_Y\circ r$ and $r \circ E = E_Y \circ q$. The triple $(\mathcal{M}/I_Y, Y, E_Y)$ is a $W^*$-bundle.   
\end{definition}

Let $\gamma\colon G \acts \mathcal{M}$ be an action on a $W^*$-bundle $(\mathcal{M},X,E)$. Since $\gamma_g(C(X)) = C(X)$ for each $g \in G$, the action can be restricted to $C(X)$ and there is an induced  action $G \acts X$ on the base space. Under the assumption that $Y \subseteq X$ is invariant under this induced action, $\gamma$ also induces an action on the quotient bundle $\mathcal{M}/I_Y$.
  
\begin{definition}[Ultrapowers of $W^*$-bundles, cf.\ {\cite[Definition 3.7]{BBSTWW}}]
Given a $W^*$-bundle $\mathcal{M}$ with uniform 2-norm $\|\cdot\|_{2,u}$, the \emph{ultrapower} of $\mathcal{M}$ is defined as
\[\mathcal{M^\omega} := \ell^\infty(\mathcal{M})/\{(a_n)_{n \in \N } \in \ell^\infty(\mathcal{M})\colon \lim_{n \rightarrow \omega} \|a_n\|_{2,u} = 0 \}.\footnote{This object can again be given the structure of a $W^*$-bundle, but for our purposes it is enough to consider only the $C^*$-algebra.}\]
Note that in the special case where $\mathcal{M}$ is the uniform tracial completion $\bar{A}^u$ of an algebraically simple, separable, nuclear $C^*$-algebra $A$ for which $T(A)$ is a non-empty Bauer simplex with finite-dimensional extremal boundary, then the definition above agrees with the ultrapower $\left(\bar{A}^u\right)^\omega$ defined in Definition \ref{definition:uniform_tracial_closure}.
\end{definition}
\begin{notation} Let $\gamma\colon G \acts \mathcal{M}$ be an action of a countable discrete group $G$ on a $W^*$-bundle $\mathcal{M}$. By Remark \ref{remark:action_W*bundles} this action preserves the $\|\cdot\|_{2,u}$-norm and hence induces an action on $\mathcal{M}^\omega$. We will denote the induced action by $\gamma^\omega$. 
\end{notation}
\begin{notation}
Given a $W^*$-bundle $(\mathcal{M},X,E)$, a non-empty subset $Y \subset X$ and $a \in \mathcal{M}$, we will denote $\sup_{x \in Y} (E(a^*a)(x))^{1/2}$ by $\|a\|_{2,Y}$.
\end{notation}
\subsection{Main technical goal and strategy}
The main goal of this section is to prove the following technical theorem:
\begin{theorem}\label{theorem:main_technical}
Let $(\mathcal{M},X,E)$ be a strictly separable $W^*$-bundle over a finite-dimensional compact metrizable space $X$, and suppose that there exists a unital $*$-homomorphism $\mathcal{R} \rightarrow \mathcal{M}^\omega \cap \mathcal{M}'$. Take $\gamma \in \op{Aut}(\mathcal{M})$ and fix $v \in \mathcal{U}(M_n)$ for some $n \in \N$. Consider the action that $\gamma$ induces on $X$. For each $x \in X$ that has finite period, we define $\mathcal{M}_{\bar{x}} := \mathcal{M}/I_{\Z \cdot x}$ and let $\gamma_{\bar{x}} \in \op{Aut}( \mathcal{M}_{\bar{x}})$ denote the automorphism induced on this quotient $W^*$-bundle
 by $\gamma$.\footnote{By the first isomorphism theorem, this quotient is isomorphic to $\pi_{\bar{x}}(\mathcal{M})$.  Here $\pi_{\bar{x}}$ denotes the GNS representation of $\mathcal{M}$ associated to the trace $\sum_{y \in \Z \cdot x}\frac{1}{|\Z \cdot x|} \tau_y$. It follows from the fact that the unit ball of $\mathcal{M}$ is $\|\cdot\|_{2,u}$-complete that this is a von Neumann algebra.} Assume that for all these $x$ with finite periods there exists a unital equivariant $*$-homomorphism \[(M_n, \op{Ad}(v))\rightarrow (\mathcal{M}_{\bar{x}}^\omega \cap \mathcal{M}_{\bar{x}}',\gamma_{\bar{x}}^\omega).\]
Then there exists a unital equivariant $*$-homomorphism \[(M_n, \op{Ad}(v)) \rightarrow (\mathcal{M}^\omega \cap \mathcal{M}',\gamma^\omega).\] 
\end{theorem}
In the special case where the action on the base space $X$ induced by $\gamma$ factors through a finite quotient of $\Z$, Theorem \ref{theorem:main_technical} follows from the work of Gardella--Hirshberg \cite{GardellaHirshberg} (they even obtained their result in the more general setting of actions of countable discrete groups instead of automorphisms).\footnote{The article \cite{GardellaHirshberg} is superseded and expanded on by a new article of Gardella--Hirshberg in collaboration with Vaccaro \cite{GardellaHirshbergVaccaro}. There, the $W^*$-bundle techniques are no longer needed and other techniques are used instead. Since we require some specific $W^*$-bundle results only found in the earlier version, we will make references to that one if needed.} This is discussed in more detail in Subsection \ref{subsection:small_periods}. 

At the opposite end, we have the case where the action induced on the base space $X$ is free. In this case the theorem can be proved by a relatively short argument combining previous results from the literature, as shown below. A sketch of this proof was provided to me by my advisor G\'abor Szab\'o, and I am thankful he let me include it here. It is again possible to do this for a more general class of countable discrete groups.

\begin{theorem}\label{theorem:free_case}
Let $(\mathcal{M},X,E)$ be a $W^*$-bundle over a finite-dimensional compact metrizable space $X$ and assume it admits a unital $*$-homomorphism $\mathcal{R} \rightarrow \mathcal{M}^\omega \cap \mathcal{M}'$. Let $G$ be a finitely generated, virtually nilpotent group. Suppose $\gamma\colon G \acts \mathcal{M}$ is an action such that the induced action on $X$ is free. Then for each $n \in \N$ and each unitary representation $\nu\colon G \rightarrow \mathcal{U}(M_n)$ there exists a unital equivariant $*$-homomorphism $(M_n, \op{Ad}(\nu)) \rightarrow (\mathcal{M}^\omega \cap \mathcal{M}', \gamma^\omega)$.  
\end{theorem}
\begin{proof}
Fix $(\mathcal{M}, X, E)$ and an action $\gamma\colon G \acts \mathcal{M}$ as in the statement of the proposition. Let $\mathcal{Q}$ denote the universal UHF algebra. First we construct a separable, $\mathcal{Q}$-stable $C^*$-subalgebra $B \subset \mathcal{M}^\omega \cap \mathcal{M}'$ that is invariant under $\gamma^\omega$. By assumption there exists a unital embedding $\mathcal{R} \rightarrow \mathcal{M}^\omega \cap \mathcal{M}'$ and hence we can find a unital embedding ${\phi_1\colon \mathcal{Q} \rightarrow \mathcal{M}^\omega \cap \mathcal{M}'}$. By definition of the $W^*$-bundle we also have a unital inclusion $C(X) \subset \mathcal{Z}(\mathcal{M}^\omega \cap \mathcal{M}')$.
 Let
\[B_1 = C^*(C(X) \cup \{\gamma^\omega_g(\phi_1(\mathcal{Q}))\mid g \in G\}).\]
Choose a unital $*$-homomorphism $\phi_2\colon \mathcal{Q} \rightarrow \mathcal{M}^\omega \cap \mathcal{M}'$ such that its image commutes with $B_1$ and set
\[B_2 = C^*(B_1 \cup\{\gamma^\omega_g(\phi_2(\mathcal{Q}))\mid g \in G\}).\]
Carry on this procedure inductively, i.e.\ given the $C^*$-algebra $B_i$ find a unital inclusion $\phi_{i+1}\colon \mathcal{Q} \rightarrow \mathcal{M}^\omega \cap \mathcal{M}'$ such that its image commutes with $B_i$ and define
\[B_{i+1} =C^*(B_i \cup\{\gamma^\omega_g(\phi_i(\mathcal{Q}))\mid g \in G\}). \]
Then set $B := \overbar{\cup_{i \in \N} B_i}$. This is separable and $\gamma^\omega$-invariant, and by construction we can find a unital inclusion $\mathcal{Q} \rightarrow B_\omega \cap B'$, which implies that $B$ is $\mathcal{Q}$-stable by \cite[7.2.2]{Rordam}. 
Denote the restriction of $\gamma^\omega$ to $B$ by $\beta$. We claim that because the action that $\gamma$ induces on $X$ is free, the restriction of $\gamma^\omega$ to $C(X)$ has finite Rokhlin dimension with commuting towers. When $G$ is finite this follows from \cite[Theorem 4.2]{Gardella} and when $G$ is countably infinite, finitely generated and virtually nilpotent, this follows from \cite[Corollary 8.5, combined with Remark 9.7]{SzaboWuZacharias}. As $C(X)$ is contained in the center of $B$, it is then immediately clear from the definition of Rokhlin dimension that also $\op{dim}_{\op{Rok}}^c(\beta) < \infty$. Now, fix $n \in \N$ and a unitary representation $\nu\colon G \rightarrow \mathcal{U}(M_n)$. By \cite[Proposition 6.3]{SsaDynSyst2}, the action ${\otimes_{\N}\op{Ad}(\nu)\colon G \acts \bigotimes_{\N} M_n}$ is strongly self-absorbing. Theorem \ref{theorem:finite_Rokhlin_dimension} (which followed from a combination of \cite[Theorem 4.10]{SzaboWuZacharias} and \cite[Corollary 5.1]{RokhlinDimension}) then shows that $\beta \otimes \left(\otimes_{\N}\op{Ad}(\nu)\right) \cong_{\mathrm{cc}} \beta$. In particular, we get a unital equivariant inclusion
$(M_n, \op{Ad}(\nu)) \rightarrow (B_\omega \cap B', \beta_\omega)$ by \cite[Corollary 3.8]{SsaDynSyst}. Since $B \subset \mathcal{M}^\omega \cap \mathcal{M}'$, a reindexation trick shows that there is a unital equivariant $*$-homomorphism $(M_n, \op{Ad}(\nu)) \rightarrow (\mathcal{M}^\omega \cap \mathcal{M}', \gamma^\omega)$. This ends the proof.
\end{proof}
By merging the approaches of the proofs for the two subcases of Theorem \ref{theorem:main_technical} in the right way, we can prove the general version of the theorem (restricting to actions of $\Z$). The proof for the free case already illustrates one of the main novel aspects of our approach: via the inclusion $C(X) \subset \mathcal{Z}(\mathcal{M}^\omega \cap \mathcal{M}')$ we will make use of Rokhlin-type properties of the dynamical system $G \acts X$ (or more precisely, a subsystem where the elements have large enough periods such that it bears enough similarity to the free case). This idea of merging the cases of small and large periods draws inspiration from a similar idea used in the proof of the main result of \cite{HirshbergWu}.
 Before the whole strategy can be explained more explicitly, we need another lemma that may be viewed as an approximate version of Lemma \ref{lemma:join_basic} for $W^*$-bundles. We use the following notation:
\begin{notation}
Consider the continuous functions $f^+, f^-\colon \R \rightarrow \R^{\geq 0}$ defined by
\[f^+(t) = \begin{cases}
t & \text{if } t \geq 0,\\
0 & \text{else,}
\end{cases}
\qquad \text{and} \qquad f^-(t) = 
\begin{cases}
-t & \text{if } t \leq 0,\\
0 & \text{else.}
\end{cases}\] 
For a self-adjoint element $a \in A$ we define two positive elements $a_+ := f^+(a)$ and $a_- := f^-(a)$ by functional calculus. Then we get $a = a_+ - a_-$ and $a_+ a_- = 0$.  
\end{notation}

\begin{lemma}\label{lemma:join}
 Let $(\mathcal{M}, X, E)$ be a strictly separable $W^*$-bundle equipped with an action $\gamma\colon G \acts \mathcal{M}$ of a countable discrete group. Let $\nu\colon G \rightarrow \mathcal{U}(M_n)$ be a unitary representation of this group. Fix a finite set $S \ssubset M_n$ generating $M_n$ linearly. Suppose that there exists an $m \in \N$ such that for every $\e >0$, every $F \ssubset G$ and $T \ssubset \mathcal{M}$ there exist $m$ c.p.c.\ order zero maps $\psi_j\colon M_n \rightarrow \mathcal{M}$ for $j=1, \hdots, m$ such that 
\begin{itemize}[itemsep=1ex,topsep=2ex]
\item $\|[\psi_i(s), \psi_j(s')]\|_{2,u} <\e$ for $s,s' \in S$ and $1 \leq i \neq j \leq m$;
\item $\|[\psi_i(s), t]\|_{2,u} <\e$ for  $s \in S$, $t \in T$ and $1 \leq i \leq m$;
\item $\|\psi_i(\op{Ad}(\nu_g)(s))- \gamma_g(\psi_i(s))\|_{2,u} < \e$ for $s \in S$, $g \in F$ and $1 \leq i \leq m$; and
\item $\|\left(1 - \sum_{i=1}^m \psi_i(1)\right)_+\|_{2,u} < \e$.
\end{itemize}
Then there exists a unital equivariant $*$-homomorphism
$\Psi\colon (M_n, \op{Ad}(\nu)) \rightarrow (\mathcal{M}^\omega \cap \mathcal{M}', \gamma^\omega)$.

In case $G$ is finitely generated, it suffices that the above holds for one finite generating set $F \subset G$ to arrive at the same conclusion.
\end{lemma}
\begin{proof}
Pick an increasing sequence of finite subsets 
$(T_k)_{k \in \N}$ whose union is dense in $\mathcal{M}$ for the $\|\cdot\|_{2,u}$-topology, and an increasing sequence of finite subsets $(F_k)_{k \in \N}$ in $G$ whose union is all of $G$. By  assumption  we can find c.p.c.\ order zero maps $\psi_j^{(k)}\colon M_n \rightarrow \mathcal{M}$ for $1 \leq j \leq m$ and every $k \in \N$ such that
\begin{itemize}[itemsep=1ex,topsep=1ex]
\item $\|[\psi^{(k)}_i(s), \psi^{(k)}_j(s')]\|_{2,u} <1/k$ for $s,s' \in S$ and $1 \leq i \neq j \leq m$;
\item $\|[\psi_i^{(k)}(s), t]\|_{2,u} <1/k$ for  $s \in S$; $t \in T_k$ and $1 \leq i \leq m$;
\item $\|\psi_i^{(k)}(\op{Ad}(\nu_g)(s))- \gamma(\psi_i^{(k)}(s))\|_{2,u} < 1/k$ for $s \in S$, $g \in F_k$ and $1 \leq i \leq m$; and
\item $\|(1 - \sum_{i=1}^m \psi_i^{(k)}(1))_+\|_{2,u} < 1/k$.
\end{itemize}
Define the maps 
\[\Psi_j\colon M_n\rightarrow \mathcal{M}^\omega \cap \mathcal{M}' \colon x \mapsto (\psi_j^{(k)}(x))_k.\]
By our assumptions these are equivariant c.p.c.\ order zero maps with pairwise commuting ranges such that 
\[\left(1-\sum_{j=1}^n \Psi_j(1)\right)_+= 0.\]
Thus, we see that $e := \sum_{j=1}^n \Psi_j(1) \geq 1$, and hence it has an inverse $e^{-1}$ of norm less than or equal to 1.  Note that for each $j=1, \hdots, m$ and $x \in M_n$ we get $[\Psi_j(x),\Psi_j(1)]=0$ by the correspondence between c.p.c.\ order zero maps and $*$-homomorphisms from cones.
Since the $\Psi_j$ are also pairwise commuting, we see that their images all commute with $e$. Moreover, $e$ is fixed by $\gamma$. 
This means that if we define maps 
\[\Psi_j'\colon M_n\rightarrow \mathcal{M}^\omega \cap \mathcal{M}'\colon x \mapsto e^{-1/2}\Psi_j(x)e^{-1/2},\]
 these will be c.p.c.\ order zero maps that are equivariant and  pairwise commuting, and such that $\sum_{j=1}^m \Psi'_j(1) = 1$. By Lemma \ref{lemma:join_basic} these give rise to a unital equivariant $*$-homomorphism $\Psi\colon (M_n,\op{Ad}(\nu)) \rightarrow (\mathcal{M}^\omega \cap \mathcal{M}',\gamma^\omega)$.  
\end{proof}

Lemma \ref{lemma:join} will be crucial to our strategy for proving Theorem \ref{theorem:main_technical}. Given an automorphism $\gamma$ of the $W^*$-bundle $\mathcal{M}$ and a unitary $v \in M_n$ as in the theorem, it suffices to find a fixed $m \in \N$ such that for each $\e >0$ there exist $m$ approximately central, approximately pairwise commuting and approximately equivariant c.p.c.\ order zero maps $\psi_j\colon M_n \rightarrow \mathcal{M}$ for $j=1, \hdots, m$
such that 
\[\left|\left|\left(1 - \sum_{i=1}^m \psi_j(1)\right)_+\right|\right|_{2,u} < \e.\]
This allows us to split the base space into two parts, $X_1$ and $X_2$, and construct the maps $\psi_j$ such that 
$ \|1-\psi_1(1)\|_{2,X_1} < \e$  and $\|1-\sum_{i=2}^m\psi_i(1)\|_{2,X_2} < \e$. In other words, this means that $\psi_1(1)$ and $\sum_{i=2}^m\psi_i(1)$ lie uniformly close to 1 on the fibers above $X_1$ and $X_2$, respectively. In particular, $X_1$ and $X_2$ can be chosen such that the behaviour of the action induced by $\gamma$ on these two sets is similar  enough to the respective two subcases already mentioned before. This allows us to obtain the maps $\psi_j$ by using similar ideas as in the proofs of these subcases.

 More precisely, we can let $X_2$ be the subset of $X$ for which the elements have periods greater than some fixed $N \in \N$ for the action induced by $\gamma$. If this $N$ is large enough, a known Rokhlin-type theorem from topological dynamics allows us to find approximate Rokhlin towers in $C_0(X_2)$. The strategy to obtain the maps $\psi_2, \hdots \psi_m$ then resembles the one used to prove the theorem in case the induced action on the base space is free. Recall from the proof of Theorem \ref{theorem:free_case} that in the free case the induced action on $C(X)$ has finite Rokhlin dimension. This by definition implies the existence of Rokhlin towers in $C(X)$, which then by the arguments in \cite{RokhlinDimension} leads to the required result. The Rokhlin towers we obtain here in $C_0(X_2)$ will play a similar role as those in the proof of \cite{RokhlinDimension}.  A more detailed explanation and proof is given in Subsection \ref{subsection:large_periods}.

The other set, $X_1 \subset X$ will consist of the elements with period at most $N$. Restricted to this set, the action of $\Z$ factors through a finite quotient. We will show that the map $\psi_1$ can be obtained from the results of Gardella--Hirshberg \cite{GardellaHirshberg}. This is what we will start with in the next subsection.  

\subsection{Traces with period $\leq N$}\label{subsection:small_periods}
In \cite{GardellaHirshberg}, Gardella and Hirshberg used $W^*$-bundle techniques to prove that equivariant $\mathcal{Z}$-stability holds automatically for actions of amenable groups on unital classifiable $C^*$-algebras whose tracial state space is a Bauer simplex with finite-dimensional extremal boundary, under the additional assumption that the induced action on the extremal boundary has finite orbits and Hausdorff orbit space. This is in particular the case if the action on the extremal boundary factors through a finite quotient group. We can use their main technical result (Theorem \ref{theorem:fiberwise_mcduff} below) to deal with the part of the $W^*$-bundle over the points of the base space whose periods are bounded above by some $N \in \N$.
Before we can state this, we need the notion of equivariant $W^*$-bundles:
\begin{definition}[{\cite[Definition 2.4]{GardellaHirshberg}}]
Let $G$ be a countable discrete group and let $X$ be a compact metrizable space. An \emph{equivariant $W^*$-bundle over $X$} is a quadruple $(\mathcal{M}, X, E, \gamma)$ such that $(\mathcal{M}, X, E)$ is a $W^*$-bundle and $\gamma\colon G \acts \mathcal{M}$ is an action satisfying $\gamma_g \circ E = E = E \circ \gamma_g$ for all $g \in G$. Note that this implies that the action restricted to $C(X)$ is the trivial one. It is easy to check that in this case $\gamma$ induces an action $\gamma_x\colon G \acts \mathcal{M}_x$ on each fiber that makes $\pi_x$, the GNS representation associated to $\tau_x$, equivariant. For this reason $\gamma$ is also called a \emph{fiberwise} action. 
\end{definition}

The following is a combination of {\cite[Theorem 2.12 and 2.20]{GardellaHirshberg}}:
\begin{theorem}\label{theorem:fiberwise_mcduff}
Let $\mathcal{M}$ be a strictly separable $W^*$-bundle over a finite-dimensional compact metrizable space $X$, let $G$ be a countable discrete group, let $\nu\colon G \rightarrow \mathcal{U}(M_n)$ be a unitary representation and let $\gamma\colon G \acts \mathcal{M}$ be a fiberwise action. Then the following are equivalent:
\begin{enumerate}
\item There exists a unital equivariant $*$-homomorphism $(M_n, \op{Ad}(\nu)) \rightarrow (\mathcal{M}^\omega \cap \mathcal{M}', \gamma^\omega)$;
\item For all $x \in X$ there exists a unital equivariant $*$-homomorphism {$(M_n, \op{Ad}(\nu)) \rightarrow (\mathcal{M}_x^\omega \cap \mathcal{M}_x', \gamma_x^\omega)$.}
\end{enumerate}
\end{theorem}

With a trick inspired by \cite[Example 2.6]{GardellaHirshberg}, we can prove a different version of the previous theorem where the action on the $W^*$-bundle no longer needs to be fiberwise, but instead it is enough that the action induced on the base space factors through a finite quotient group:

\begin{theorem}\label{theorem:GH_rewrite}
Let $(\mathcal{M},X,E)$ be a strictly separable $W^*$-bundle over a finite-dimensional compact metrizable space $X$. Let $G$ be a countable discrete group with a unitary representation $\nu\colon G \rightarrow \mathcal{U}(M_n)$, and take an action $\gamma\colon G \acts \mathcal{M}$ such that the induced action on $X$ factors through a finite quotient of $G$. For $x \in X$ define $\mathcal{M}_{\bar{x}} := \mathcal{M}/I_{G \cdot x}$ and let $\gamma_{\bar{x}}\colon G \acts \mathcal{M}_{\bar{x}}$ denote the action that $\gamma$ induces on this quotient bundle. Assume that for all $x \in X$ there exists a unital equivariant $*$-homomorphism
\[(M_n, \op{Ad}(\nu)) \rightarrow (\mathcal{M}_{\bar{x}}^\omega \cap \mathcal{M}_{\bar{x}}', \gamma_{\bar{x}}^\omega).\]
Then there exists a unital equivariant $*$-homomorphism \[(M_n, \op{Ad}(\nu)) \rightarrow (\mathcal{M}^\omega \cap \mathcal{M}', \gamma^\omega).\]
\end{theorem}
\begin{proof}
Take a $W^*$-bundle $(\mathcal{M},X,E)$ as in the statement of the theorem. In order to be able to apply Theorem \ref{theorem:fiberwise_mcduff} we first turn this into an equivariant $W^*$-bundle with the same trick as used in \cite[Example 2.6]{GardellaHirshberg}. Since the action that $\gamma$ induces on $X$ factors through a finite quotient of $G$, the orbit space $X/G$ is again compact and metrizable, and there is a canonical faithful condition expectation $\mathcal{E}\colon C(X) \rightarrow C(X/G)$ given by
\[\mathcal{E}(f)(G \cdot x) = \frac{1}{|G \cdot x|} \sum_{y \in G \cdot x} f(y).\] 
We show that $\mathcal{M}$ also has a natural structure of an equivariant $W^*$-bundle over $X/G$ with conditional expectation $E' =  \mathcal{E} \circ E$. Since $\gamma_g \circ E' = E' = E' \circ \gamma_g$ for each $g \in G$, it will moreover be an equivariant bundle. Let $N$ denote an upper bound for the orbit size. For each $a \in \mathcal{M}$ it holds that
\[\sup_{x \in X} E(a^*a)(x) \leq N \sup_{G \cdot x \in X/G} E'(a^*a)(G \cdot x) \leq N\sup_{x \in X} E(a^*a)(x).\]
This means that the two possible uniform $2$-norms on $\mathcal{M}$ defined respectively by $E$ and $E'$ are equivalent. In particular, the unit ball of $\mathcal{M}$ is also complete for the $2$-norm induced by $E'$, and it does not matter which $2$-norm we use to form the ultrapower $\mathcal{M}^\omega$, as the two options will yield the same $C^*$-algebra. 

By the first isomorphism theorem, for each $x \in X$ the fiber of $\mathcal{M}$ over $G \cdot x$ is isomorphic to $\mathcal{M}_{\bar{x}} = \mathcal{M}/I_{G \cdot x}$ as defined in the statement of the theorem. Since by assumption there exists a unital equivariant $*$-homomorphism 
\[(M_n, \op{Ad}(\nu)) \rightarrow (\mathcal{M}_{\bar{x}}^\omega \cap \mathcal{M}_{\bar{x}}', \gamma_{\bar{x}}^\omega)\]
for each $x \in X$, the result follows from Theorem \ref{theorem:fiberwise_mcduff}.  
\end{proof}

Consider the case of a $W^*$-bundle $(\mathcal{M},X,E)$ as in the statement of Theorem \ref{theorem:main_technical} with an automorphism $\gamma \in \op{Aut}(\mathcal{M})$ and $v \in \mathcal{U}(M_n)$. The action induced by $\gamma$ on the base space $X$ does not need to factor through a finite quotient of $\Z$ and therefore we cannot immediately apply the previous theorem. However, given some fixed $N \in \N$ we can consider the subset $X_N \subset X$ of points in the base space that have period at most $N$ for this induced action. The quotient bundle $\mathcal{N} :=  \mathcal{M}/I_{X_N}$ with the induced automorphism $\bar{\gamma} \in \op{Aut}(\mathcal{N})$ will satisfy the requirements of Theorem \ref{theorem:GH_rewrite} and hence, we get a unital equivariant map 
 \[(M_n, \op{Ad}(v)) \rightarrow (\mathcal{N}^\omega \cap \mathcal{N}', \bar{\gamma}^\omega).\]
 The next proposition will allow us to lift this map to a c.p.c.\ map into $\mathcal{M}^\omega \cap \mathcal{M}'$. 
\begin{prop}\label{prop:isomoprhism_ultrapowers}
Let $(\mathcal{M},X,E)$ be a strictly separable $W^*$-bundle over a finite-dimensional compact metrizable space $X$ and let $\gamma \in \op{Aut}(\mathcal{M})$. Fix $N \in \N$ and denote by $X_N \subset X$ the points of the base space with period at most $N$ for the action induced by $\gamma$. Define the ideal $J_{X_N} \subset \mathcal{M}^\omega$ as (abusing notation to use representative sequences in $\ell^\infty(\mathcal{M})$ to denote elements in $\mathcal{M}^\omega$)
\[J_{X_N} := \left\{(a_n)_{n \in \N} \in \mathcal{M}^\omega \colon \lim_{n \rightarrow \omega} \|a_n\|_{2,X_N} =0\right\}.\] 
This is a $\Z$-$\sigma$-ideal. Consider the quotient $W^*$-bundle $\mathcal{N} := \mathcal{M}/I_{X_N}$.
 Then as $C^*$-algebras, we get that
\[\mathcal{N}^\omega \cong \mathcal{M}^\omega/J_{X_N}.\]
\end{prop}
\begin{proof}
Recall Example \ref{example:sigma_ideal}, where we sketched a proof of the fact that in presence of an action by a group $G$, the trace kernel ideal of a $C^*$-algebra is a $G$-$\sigma$-ideal of its norm ultrapower. The proof that $J_{X_N}$ is a $\Z$-$\sigma$-ideal for the action induced by $\gamma$ is completely analogous. 

Let $\pi\colon \mathcal{M} \rightarrow \mathcal{N}$ denote the projection map to the quotient bundle. Recall that the conditional expectation $E_\mathcal{N}$ on $\mathcal{N}$ satisfies $E_\mathcal{N} \circ \pi = r \circ E$, where $r\colon C(X) \rightarrow C(X_N)$ denotes the restriction map. Hence, for all $a \in \mathcal{M}$ it holds that
\[\|\pi(a)\|_{2,u}=  \|a\|_{2,X_N}.\]

The quotient map $\pi$ induces a surjective map $\ell^\infty(\mathcal{M}) \rightarrow \ell^\infty(\mathcal{N})$ and the above equality shows that under this map the ideal
$J_{X_N}$ is the preimage of the ideal
\[\{(a_n)_{n\in \N} \in \ell^\infty(\mathcal{N}) \colon \lim_{n \rightarrow \omega}\|a_n\|_{2,u} =0\} \subset \ell^\infty(\mathcal{N}).\]
Hence, 
 \[\mathcal{N}^\omega \cong \mathcal{M}^\omega/J_{X_N}.\]
\end{proof}

\begin{corollary}\label{corollary:small_periods}
Let $\mathcal{M}$ be a strictly separable $W^*$-bundle over a finite-dimensional compact metrizable space $X$ and fix $\gamma \in \op{Aut}(\mathcal{M})$. Take $v \in \mathcal{U}(M_n)$ for some $n \in \N$. For any $x \in X$ that has finite period for the action that $\gamma$ induces on $X$, we define $\mathcal{M}_{\bar{x}} := \mathcal{M}/I_{\Z \cdot x}$ and write $\gamma_{\bar{x}} \in \op{Aut}(\mathcal{M}_{\bar{x}})$ for the automorphism that $\gamma$ induces on this quotient $W^*$-bundle. Assume that for all those $x$ there exists a unital equivariant $*$-homomorphism 
\begin{equation}\label{eq:orbitwise_absorption}
(M_n, \op{Ad}(v))\rightarrow (\mathcal{M}_{\bar{x}}^\omega \cap \mathcal{M}_{\bar{x}}',\gamma_{\bar{x}}^\omega).
\end{equation}
Then, using the same notation as in Proposition \ref{prop:isomoprhism_ultrapowers}, there exists an equivariant c.p.c.\ order zero map \[\Phi\colon (M_n, \op{Ad}(v)) \rightarrow (\mathcal{M}^\omega \cap \mathcal{M}',\gamma^\omega)\] such that $1-\Phi(1) \in J_{X_N}$.
\end{corollary}
\begin{proof}
Denote the quotient $W^*$-bundle $\mathcal{M}/I_{X_N}$ by $\mathcal{N}$ and the induced automorphism on this bundle by $\bar{\gamma}$. Because of assumption \eqref{eq:orbitwise_absorption} on $\mathcal{M}$, the quotient bundle $\mathcal{N}$ satisfies all the necessary properties for the conclusion of Theorem \ref{theorem:GH_rewrite} to hold, so there must exist a unital equivariant map $(M_n, \op{Ad}(v)) \rightarrow (\mathcal{N}^\omega \cap \mathcal{N}', \bar{\gamma}^\omega)$. Since $\mathcal{M}$ is strictly separable, a combination of Proposition \ref{prop:sigma_ideal} and \ref{prop:isomoprhism_ultrapowers} allows us to lift this map to an equivariant c.p.c.\ order zero map 
\[\Phi\colon (M_n, \op{Ad}(v)) \rightarrow (\mathcal{M}^\omega \cap \mathcal{M}',\gamma^\omega)\] such that $1-\Phi(1) \in J_{X_N}$.
\end{proof}

\subsection{Traces with period $> N$}\label{subsection:large_periods}
Next, we show that we can also take care of traces with sufficiently large periods: 

\begin{lemma}\label{lemma:large_period}
Let $(\mathcal{M},X,E)$ be a $W^*$-bundle over a compact metrizable space $X$ that has covering dimension $d < \infty$. Assume that there exists a unital $*$-homomorphism $\mathcal{R} \rightarrow \mathcal{M}^\omega \cap \mathcal{M}'$. Fix $\gamma \in \op{Aut}(\mathcal{M})$ and $v \in \mathcal{U}(M_n)$ for some $n \in \N$. For each $\e>0$ and $S \ssubset M_n$ there exists an $N \in \N$ such that the following holds: Let $X_N \subset X$ denote those points with period at most $N$ for the action that $\gamma$ induces on $X$. Let $K \subset X \setminus X_N$ be compact. For each $T \ssubset \mathcal{M}$ we can find $2d+3$ c.p.c.\ order zero maps $\psi_l\colon M_n \rightarrow \mathcal{M}$ for $l=0, \hdots 2d+2$ such that for all $t \in T$ and $s,s' \in S$
\begin{itemize}[itemsep=1ex,topsep=1ex]
\item
$\|\psi_l(vsv^*) - \gamma(\psi_l(s))\|_{2,u} < \e \quad \text{for } l=0, \hdots 2d+2;$
\item
$\| [\psi_l(s),\psi_{l'}(s')]\|_{2,u} < \e \quad \text{for } 0 \leq l \neq l' \leq 2d+2; $
\item $\|[\psi_{l}(s), t]\|_{2,u}  < \e \quad \text{for } l=0, \hdots 2d+2$; and
\item $\left|\left|\sum_{l=0}^{2d+2} \psi_l(1) - 1\right|\right|_{2,K} =0.$
\end{itemize}
\end{lemma}

Before giving the detailed proof, we explain the idea behind it. By the assumptions on $\mathcal{M}$ we can find tracially-approximate $*$-homomorphisms from $M_n$ into $\mathcal{M}$  that are approximately central. The trick is to adjust and average these in the right way so that they become c.p.c.\ order zero maps with the needed properties as specified in the proposition. Since $\mathcal{M}$ is a $W^*$-bundle over $X$, we have a canonical inclusion $C(X) \subset Z(\mathcal{M})$ by definition. This allows us to use results about the dynamical system $X$ equipped with the action it gets from $\gamma$ (or better: the subsystem of elements with large enough period) in order to obtain the right elements to average over. 
More specifically, we need the following Rokhlin-type lemma from the context of topological dynamics:

\begin{lemma}[{\cite[Lemma 4.2]{HirshbergWu}}]\label{lemma:Rokhlin}
Let $X$ be a locally compact space with covering dimension $d < \infty$. Fix $k,m \in \N$ and a compact subset $K \subset X$, and suppose $\alpha\colon \Z \acts X$ is an action by homeomorphisms such that $|\Z \cdot x| > (d+1)(4m+1)$ for any $x \in X$. 
\begin{enumerate}
\item If $m \geq (2d+3)k-\frac{d}{2}-1$, then there exist open subsets $Z^{(0)}, \hdots, Z^{(2d+2)} \subset X$ such that
\begin{enumerate}[itemsep=1ex]
\item for any $l \in {\{0, \hdots, 2d+2\}}$, the sets $\alpha_{-m}(\overbar{Z}^{(l)}), \hdots, \alpha_{0}(\overbar{Z}^{(l)}), \hdots, \alpha_{m}(\overbar{Z}^{(l)})$ are pairwise disjoint;
\item $ K \subset \bigcup_{l=0}^{2d+2} \bigcup_{i=-(m-k)}^{m-k}\alpha_i(Z^{(l)}).$
\end{enumerate}
\item If $\e >0$ satisfies $m \geq (2d+3)k\left\lceil \frac{1}{\e}\right \rceil - \frac{d}{2}-1$, there exist open subsets $Z^{(0)}, \hdots, Z^{(2d+2)} \subset X$ satisfing the above two conditions, as well as compactly supported contractions $\{f_j^{(l)}\}_{l \in \{0, \hdots 2d+2\}; j \in \Z} \subset C(X)_+$ satisfying
\begin{enumerate}[itemsep=1ex]
\item $\textup{supp}(f_j^{(l)}) \subseteq 
\begin{cases}
\alpha_j(Z^{(l)}), & \text{if } |j| \leq m,\\
\emptyset, & \text{if } |j| > m\\
\end{cases}$ for any $l \in \{0, \hdots, 2d+2\};$
\item $\sum_{l=0}^{2d+2} \sum_{j=-m}^m f_j^{(l)}(x) \leq 1$ for all $x \in X$ with equality on $K$;
\item $\|f_j^{(l)} \circ \alpha_i - f_{j-i}^{(l)}\| < \e$ for all $j \in \Z$, for all $i \in \Z \cap [-k,k]$ and for all $l \in \{0,\hdots, 2d+2\}$.
\end{enumerate}
\end{enumerate}
\end{lemma}

\begin{proof}[Proof of Lemma \ref{lemma:large_period}]
Let $(\mathcal{M},X,E)$ be a $W^*$-bundle satisfying the requirements of the proposition, with $d:= \op{dim}(X) < \infty$. Let $\gamma \in \op{Aut}(\mathcal{M})$. Fix some $v \in \mathcal{U}(M_n)$, $\e >0$ and $S \ssubset M_n$. We may assume that $\max_{s \in S}\{\|s\|\} \leq 1$. Let $\delta$ denote the automorphism $\op{Ad}(v)$ on $M_n$.
Choose
\[m \geq(2d+3)\left\lceil\frac{2}{\e} \right\rceil - \frac{d}{2} -1 \text{ and }\]
\[N > (d+1)(4m+1).\]
We show that this $N$ satisfies the requirements of the proposition.

Take $T \ssubset \mathcal{M}$ arbitrarily. 
By the assumptions on $\mathcal{M}$, we can find $2d+3$ unital $*$-homomorphisms ${\Phi_l\colon M_n \rightarrow \mathcal{M}^\omega \cap \mathcal{M}'}$ for $l=0, \hdots, 2d+2$ such that 
\[[\Phi_l(M_n), \gamma^z(\Phi_{l'}(M_n))]=0 \quad \text{for }l \neq l', z \in \Z.\]
By \cite[Proposition 3.11]{BBSTWW} there exists as unital $*$-homomorphism $\iota\colon M_n \rightarrow \mathcal{M}$, and this induces a unital $*$-homomorphism $\iota^\omega\colon M_n \rightarrow \mathcal{M}^\omega$. 
As unital $*$-homomorpisms $M_n \to \mathcal{M}^\omega$ are unique up to unitary equivalence, the maps $\Phi_l$ can be lifted to a sequence of unital $*$-homomorphisms into  $\mathcal{M}$. Indeed, given $l =0, \hdots, 2d+2$, we can find a unital $*$-homomorphism $\theta\colon M_n \rightarrow \mathcal{M}^\omega$ whose image commutes with the images of $\iota^\omega$ and $\Phi_l$ by the assumptions on $\mathcal{M}$. Let $u \in \mathcal{U}(M_n \otimes M_n)$ be a unitary implementing the flip automorphism on $M_n \otimes M_n$. Then 
\[w = (\Phi_l \otimes \theta)(u)(\iota^\omega \otimes \theta)(u) \in \mathcal{U}(\mathcal{M}^\omega)\] satisfies $\mathrm{Ad}(w) \circ \iota^\omega = \Phi_l$. As $w$ has finite spectrum it can be lifted to a sequence of unitaries $(w_n)_{n \in \N}$ in $\mathcal{M}$, and then the sequence of $*$-homomorphisms $(\mathrm{Ad}(w_n) \circ \iota)_{n \in \N}$ lifts $\Phi_l$. In particular, we can find unital $*$-homomorphisms ${\phi_l\colon M_n \rightarrow \mathcal{M}}$ for $l=0, \hdots, 2d+2$ such that for all $s \in S$, $t \in T$ and $i,j = -m,\hdots, m$ it holds that
\begin{align}
\|[\gamma^i(\phi_l(\delta^{-i}(s))), \gamma^j(\phi_{l'}(\delta^{-j}(s)))]\|_{2,u} &< \e \quad \text{ if } l \neq l', \text{ and } \label{eq:commutativity}\\
\|[\gamma^i(\phi_l(\delta^{-i}(s))), t]\|_{2,u} &< \e. \label{eq:centrality}
\end{align}

Let $X'\subset X$ denote the points with period greater than $N$ for the action that $\gamma$ induces on $X$ and take a compact set $K \subset X'$.
 Recall that $N > (d+1)(4m+1)$ and thus, we can apply Lemma \ref{lemma:Rokhlin} to the action that $\gamma$ induces on $X'$ and find compactly supported contractions $\{f_j^{(l)}\} \subset C(X')_+$ for $l=0, \hdots 2d+2$ and $j \in \Z$ satisfying
\begin{align}
\op{supp}\left(f_j^{(l)}\right) &= \emptyset \text{ for } |j| >m; \label{eq:support}\\
\gamma^i\left(f_j^{(l)}\right) \gamma^{i'}\left(f_{j'}^{(l)}\right) &=0 \text{ if } i+j \neq i'+j' \text{ and } i+j-(i'+j') \leq 2m;\label{eq:orthogonality}\\
\sum_{l=0}^{2d+2} \sum_{j=-m}^m f_j^{(l)}(x) &\leq 1 \text{  for all } x \in X', \text{ with equality in } K;\label{eq:joint_unitaliy}\\
\left|\left|\gamma \left( f_j^{(l)} \right) - f_{j+1}^{(l)}\right|\right| &< \e/2 \text{ for } l= 0, \hdots, 2d+2 \text{ and all } j \in \Z \label{eq:action}.
\end{align}
Note that it makes sense to apply $\gamma$ to the elements $f_j^{(l)}$. Since $X'$ is open and all $f_j^{(l)}$ are compactly supported, the $f_j^{(l)}$ can be extended to continuous functions on $X$ by setting them to be zero outside $X'$. In this way they can be considered as elements in $C(X) \subset Z(\mathcal{M})$. Define the maps
$\psi_l\colon M_n \rightarrow \mathcal{M}$ for $l=0, \hdots, 2d +2 $ by
\[\psi_l(z) :=  \sum_{j=-m}^m f_j^{(l)} \gamma^j(\phi_l(\delta^{-j}(z))) \quad \text {for } z \in M_n.\] 
These are clearly c.p.c.\ maps and they are also order zero since $f_i^{(l)}f_j^{(l)}=0$ if $i,j =-m, \hdots, m$ with $i \neq j$ by \eqref{eq:orthogonality}.
Fix $x \in K$. Recall that the trace $\tau_x$ on $\mathcal{M}$ is defined by composing the conditional expectation $E$ with evaluation at $x$. We get 
\begin{align*}\textstyle\|\sum_{l=0}^{2d+2}\psi_l(1) - 1\|_{2,\tau_x}
&= \textstyle\|\sum_{l=0}^{2d+2} \sum_{j=-m}^m f_j^{(l)} - 1\|_{2,\tau_x}\\
\overset{\eqref{eq:joint_unitaliy}}&{=} 0.
\end{align*}
When $l \neq l' $ and $s, s' \in S$ we get
\begin{align*}
\|[\psi_l(s),\psi_{l'}(s')]\|_{2,u}&= \|[ \sum_{j=-m}^m f_j^{(l)} \gamma^j(\phi_l(\delta^{-j}(s))),  \sum_{j=-m}^m f_j^{(l')} \gamma^j(\phi_{l'}(\delta^{-j}(s')))]\|_{2,u}\\
&=\| \sum_{i=-m}^m\sum_{j=-m}^m f_i^{(l)}f_j^{(l')}[ \gamma^i(\phi_l(\delta^{-i}(s))),  \gamma^j(\phi_{l'}(\delta^{-j}(s')))]\|_{2,u}\\
&\leq \sup_{x \in X} \sum_{i=-m}^m\sum_{j=-m}^m f_i^{(l)}(x)f_j^{(l')}(x)\|[ \gamma^i(\phi_l(\delta^{-i}(s))),  \gamma^j(\phi_{l'}(\delta^{-j}(s')))]\|_{2,\tau_x}\\
\overset{\eqref{eq:commutativity}}&{<}\sup_{x\in X} \sum_{i=-m}^m\sum_{j=-m}^m f_i^{(l)}(x)f_j^{(l')}(x) \e \overset{\eqref{eq:joint_unitaliy}}{\leq}\e.
\end{align*}
Similarly, using \eqref{eq:centrality} we also see that $\|[\psi_l(s),t]\|_{2,u} < \e$ for $l = 0, \hdots, 2d+2$, $t\in T$ and $s \in S$. 

Next, we show that these maps are approximately equivariant in the right sense. For $s \in S$ we get
\begin{align*}
\gamma(\psi_l(s)) &= \sum_{j=-m}^m \gamma(f_j^{(l)}) \gamma^{j + 1}(\phi_l(\delta^{-j}(s)))\\
&= \sum_{i=-m+1}^{m+1} \gamma(f_{i-1}^{(l)})\gamma^i(\phi_l(\delta^{-i+1}(s)))\\
\overset{\eqref{eq:support}}&{=} \sum_{i=-m}^{m+1} (\gamma(f_{i-1}^{(l)}) - f_i^{(l)}) \gamma^i(\phi_l(\delta^{-i+1}(s))) +\sum_{i=-m}^m f_i^{(l)}\gamma^i(\phi_l(\delta^{-i+1}(s)))\\
&=\sum_{i=-m}^{m+1} (\gamma(f_{i-1}^{(l)}) - f_i^{(l)}) \gamma^i(\phi_l(\delta^{-i+1}(s))) + \psi_l(\delta^1(s)).\\
\end{align*}
So, we get
\begin{align*}
\|\gamma(\psi_l(s)) -\psi_l(vsv^*)\|_{2,u}&= \|\sum_{i=-m}^{m+1} (\gamma(f_{i-1}^{(l)}) - f_i^{(l)}) \gamma^i(\phi_l(\delta^{-i+1}(s)))\|_{2,u}\\
&\leq  \|\sum_{i=-m}^{0} (\gamma(f_{i-1}^{(l)}) - f_i^{(l)}) \gamma^i(\phi_l(\delta^{-i+1}(s)))\|_{2,u} \\ & \quad+ \|\sum_{i=1}^{m+1} (\gamma(f_{i-1}^{(l)}) - f_i^{(l)}) \gamma^i(\phi_l(\delta^{-i+1}(s)))\|_{2,u}\\
\overset{\eqref{eq:orthogonality}, \eqref{eq:action}}&{<}\e/2+\e/2 = \e.
\end{align*}
As we have shown that the maps $\psi_l$ satisfy all the necessary properties, this ends the proof.
\end{proof}
\subsection{Proof of Theorem \ref{theorem:main_technical}}
Combining the results from the previous two subsections, we can prove the main technical result of this paper. Before doing this, we need one more small lemma: 

\begin{lemma}\label{lemma:tracial_plus}
Let $A$ be a $C^*$-algebra and $\tau \in T(A)$. Suppose $a$ is a self-adjoint element and $b \geq 0$. Then 
\[\|(a-b)_+\|_{2,\tau} \leq \|a\|_{2,\tau}.\]
\end{lemma}
\begin{proof}
We have
\[a = (a-b)_+ - (a-b)_- +b,\]
so 
\begin{align*} a^2 &= (a-b)^2_+ + (a-b)^2_- +b^2 + b(a-b)_+ - b(a-b)_- +(a-b)_+b - (a-b)_-b\\
&=(a-b)^2_+ + (a-b)^2_- +b^2 + b(a-b) +(a-b)b.
\end{align*}
Since $(a-b) \geq - (a-b)_-$, we also get that 
\begin{align*}
\tau((a-b)^2_- +b^2 + b(a-b) +(a-b)b) &=\tau((a-b)^2_- +b^2 + 2b^{1/2}(a-b)b^{1/2}) \\
&\geq \tau((a-b)^2_- +b^2 - 2b^{1/2}(a-b)_-b^{1/2})\\
&= \tau([b-(a-b)_-]^2) \geq 0.
\end{align*}
Hence, 
\begin{align*}
\|(a-b)_+\|_{2,\tau} &= \tau((a-b)^2_+)\\
& \leq \tau((a-b)^2_+ + (a-b)^2_- +b^2 + b(a-b) +(a-b)b)\\
&= \tau(a^2) \\ &= \|a\|_{2,\tau}.
\end{align*}
\end{proof}

\begin{proof}[Proof of Theorem {\ref{theorem:main_technical}}]
Fix a finite subset $S \ssubset M_n$ generating $M_n$ linearly. Denote $d := \op{dim}(X) < + \infty$. Note that by Lemma \ref{lemma:join} it suffices to show that for  
any $\e >0$ and $T \ssubset \mathcal{M}$ there exists $2d+4$ c.p.c.\ order zero maps $\psi_l\colon M_n \rightarrow \mathcal{M}$ for $l=0, \hdots, 2d+3$ that satisfy
\begin{itemize}[itemsep=1ex, topsep=1ex]
\item $\|[\psi_i(s), \psi_j(s')]\|_{2,u} <\e$ for $s,s' \in S$ and $0 \leq i \neq j \leq 2d+3$;
\item $\|[\psi_i(s), t]\|_{2,u} <\e$ for $s \in S$, $t \in T$ and $i=0, \hdots, 2d+3$;
\item $\|\psi_i(vsv^*)- \gamma(\psi_i(s))\|_{2,u} < \e$ for $i =0, \hdots, 2d+3$; and
\item $\|\left(1 - \sum_{i=1}^m \psi_i(1)\right)_+\|_{2,u} < \e$.
\end{itemize}
So, fix $\e >0$ and $T \ssubset \mathcal{M}$.
Fix the $N \in \N$ from Lemma \ref{lemma:large_period} that corresponds to the tuple $(\e,S)$. By Corollary \ref{corollary:small_periods} (using the same notation as in the statement of that corollary) we know that there is an equivariant c.p.c.\ order zero map 
\[\Psi_0\colon (M_n,\op{Ad}(v)) \rightarrow (\mathcal{M}^\omega \cap \mathcal{M}',\gamma^\omega) \quad \text{such that } 1- \Psi_0(1) \in J_{X_N}.\] 
Since order zero maps from matrix algebras always lift, we can lift $\Psi_0$ to a sequence of maps into $\mathcal{M}$, and going far enough in the sequence allows us to obtain a c.p.c.\ order zero map $\psi_{0}\colon M_n \rightarrow \mathcal{M}$ satisfying
\begin{itemize}[itemsep=1ex,topsep=1ex]
\item $\|[\psi_0(s), t]\|_{2,u} <\e$ for $s \in S$ and $t \in T$; 
\item $\|\psi_0(vsv^*)- \gamma(\psi_0(s))\|_{2,u} < \e$ for $s \in S$; and
\item $\|1 -  \psi_0(1)\|_{2,X_N} < \e$.
\end{itemize}
By continuity we can find an open subset $V \subset X$ containing $X_N$ such that 
\begin{equation}\label{eq:almost_unital_on_V}\|1 -  \psi_0(1)\|_{2,V} < \e.\end{equation} 
Then the complement $X \setminus V$ is compact and all points there have period greater than $N$, so by Lemma \ref{lemma:large_period} we find  $2d+3$ c.p.c.\ order zero maps $\psi_l\colon M_n \rightarrow \mathcal{M}$ for $l=1, \hdots 2d+3$ such that for all  $t' \in (T \cup \psi_0(S))$ and $s,s' \in S$ it holds that
\begin{itemize}[itemsep=1ex,topsep=1ex]
\item $\|\psi_l(vsv^*) - \gamma(\psi_l(s))\|_{2,u} < \e$ for $l=1, \hdots 2d+3$;
\item 
$\| [\psi_l(s),\psi_{l'}(s')]\|_{2,u} < \e$ for  $1 \leq l \neq l' \leq 2d+3$;
\item
$\|[\psi_{l}(s), t']\|_{2,u}  < \e$ for $ 1 \leq l \leq 2d+3$; and
\item
$ \|1 - \sum_{l=1}^{2d+3} \psi_l(1)\|_{2,X \setminus V} < \e.$
\end{itemize}
Combining the last estimate with \eqref{eq:almost_unital_on_V} and applying Lemma \ref{lemma:tracial_plus} (using $a = 1-\psi_0(1)$, $b = \sum_{l=1}^{2d+3} \psi_l(1)$ for traces in $V$ and $a= 1-\sum_{l=1}^{2d+3} \psi_l(1)$, $b= \psi_0(1)$ for traces in $X \setminus V$), we conclude that 
\[ \left|\left|\left(1-\sum_{l=0}^{2d+3} \psi_l(1)\right)_+\right|\right|_{2,u} < \e.\]
All necessary conditions to apply Lemma \ref{lemma:join} are fulfilled. This ends the proof. 
\end{proof}

\section{Main results}\label{section:main}
In this last section, we return to actions on an algebraically simple, separable, nuclear, $\mathcal{Z}$-stable $C^*$-algebra $A$ for which $T(A)$ is a non-empty Bauer simplex with finite-dimensional extremal boundary. In this case, the uniform tracial completion $\bar{A}^u$ is isomorphic to the trivial $W^*$-bundle $C_\sigma(\partial_eT(A), \mathcal{R})$, so we can apply the results of the previous section to obtain structural results about actions on $A$.

\begin{theorem}\label{theorem:main}
Let $A$ be an algebraically simple, separable, nuclear, $\mathcal{Z}$-stable $C^*$-algebra. Assume that $T(A)$ is a non-empty Bauer simplex and that $\partial_e T(A)$ is finite-dimensional, and let $\alpha\colon \Z \acts A$ be an action on $A$. Then the following hold:
\begin{enumerate}
\item For each $n \in \N$, there exists a unital $*$-homomorphism $M_n \rightarrow (A^\omega \cap A')^{\alpha^\omega}$;
\item If $\alpha$ is strongly outer, then for any $n \in \N$ and unitary representation $\nu\colon \Z \rightarrow \mathcal{U}(M_n)$ there exists a unital equivariant $*$-homomorphism $(M_n, \op{Ad}(\nu)) \rightarrow (A^\omega \cap A', \alpha^\omega)$. 
\end{enumerate}
\end{theorem}

\begin{proof}
We will prove this using Theorem \ref{theorem:main_technical}. Take a $C^*$-algebra $A$ satisfying the assumptions above. Recall from Example \ref{example:W-bundles} that $\mathcal{M}:= \bar{A}^u$ is isomorphic to the trivial $W^*$-bundle $C_\sigma(\partial_e T(A), \mathcal{R})$. Furthermore, the unit ball of $A$ is $\|\cdot\|_{2,u}$-dense in the unit ball of $\bar{A}^u$ and hence $A^\omega\cap A' \cong \mathcal{M}^\omega \cap \mathcal{M}'$. We will denote the extension of $\alpha$ to $\mathcal{M}$ by $\gamma$.

The $W^*$-bundle $\mathcal{M}$ is strictly separable because $A$ itself is separable. Since it is a trivial $W^*$-bundle with fibers isomorphic to $\mathcal{R}$, there exists a unital $*$-homomorphism $\mathcal{R} \rightarrow \mathcal{M}^\omega \cap \mathcal{M'}$. 
Now, take any $\tau \in \partial_e T(A)$ that has finite orbit for the action induced on $\partial_e T(A)$ and consider the trace 
\[\bar{\tau} := \frac{1}{|\Z \cdot \tau|}\sum_{\sigma \in \Z \cdot \tau} \sigma.\]
Let $\pi_{\bar{\tau}}$ denote the GNS representation of $\mathcal{M}$ associated to $\bar{\tau}$. The kernel of this map is $I_{\Z \cdot \tau}$ and hence 
\begin{equation}\label{eq:isomorphism_GNS}
\mathcal{M}/I_{\Z \cdot \tau} \cong \pi_{\bar{\tau}}(\mathcal{M}) \cong \pi_{\bar{\tau}}(A)''.
\end{equation}
 Denote by $\alpha_{\bar{\tau}}$ the action on $\pi_{\bar{\tau}}(A)''$ induced by $\alpha$ that makes $\pi_{\bar{\tau}}$ equivariant. Under the isomorphism \eqref{eq:isomorphism_GNS}, $\alpha_{\bar{\tau}}$ corresponds to the action that $\gamma$ induces on the quotient $\mathcal{M}/I_{\Z \cdot \tau}$.

 Take a unitary representation $\nu\colon \Z \rightarrow \mathcal{U}(M_n)$. We wish to determine when it is possible to find a unital equivariant $*$-homomorphism $(M_n, \op{Ad}(\nu)) \rightarrow (A^\omega \cap A', \alpha^\omega)$. By Theorem \ref{theorem:main_technical} and all the correspondences previously established, it suffices to show that for any 
$\tau \in \partial_e T(A)$ that has finite orbit for the action induced on $\partial_e T(A)$, there exists a unital equivariant $*$-homomorphism
\begin{equation}\label{eq:absorption_sum_fibers}(M_n, \op{Ad}(\nu)) \rightarrow \left((\pi_{\bar{\tau}}(A)'')^\omega \cap \pi_{\bar{\tau}}(A)', (\alpha_{\bar{\tau}})^\omega\right).
\end{equation}

For all such $\tau \in \partial_e T(A)$ it holds that  
$\pi_{\bar{\tau}}(A)'' \cong \bigoplus_{\sigma \in \Z \cdot \tau} \pi_\sigma(A)''$ (see \cite[Chapter 6]{Dixmier}). Since $\tau$ is extremal all direct summands are factors and because $A$ is nuclear, it follows from Connes' theorem (see \cite{Connes76,ChoiEffros}) that \[ \pi_{\bar{\tau}}(A)'' \cong \bigoplus_{\sigma \in \Z \cdot \tau}  \mathcal{R}.\]
In case $\nu$ is the trivial representation, the action $\op{Ad}(\nu)$ on $M_n$ is the trivial action and it follows directly from 
\cite[Proposition 5.19]{EquivariantSI} that \eqref{eq:absorption_sum_fibers} holds for all extremal traces with finite period. This proves (1). 

Now assume we are in case (2), i.e.\ $\alpha$ is strongly outer. Let $\nu$ be an arbitrary unitary representation. Fix $\tau \in \partial_e T(A)$ with finite period. Since $\bar{\tau}$ is invariant under $\alpha$, by definition of strong outerness we get that $\alpha_{\bar{\tau}}$ is pointwise outer, and even properly outer (see \cite[Remark 2.17]{GardellaLupini18}). Define the action $\delta\colon \Z  \acts  \mathcal{R}$ by extending ${\bigotimes_{n=1}^\infty \op{Ad}(\nu)\colon \Z \acts \bigotimes_{\N} M_n}$ to $\mathcal{R}$. Then \cite[Proposition 5.19]{EquivariantSI} shows that there exists a unital equivariant embedding of  $(\mathcal{R}, \delta)$ into the tracial central sequence algebra of $\pi_{\bar{\tau}}(A)''$, and thus also an embedding of $(M_n, \op{Ad}(\nu))$. Again, condition \eqref{eq:absorption_sum_fibers} is satisfied for all relevant traces. This proves (2).
\end{proof}

From the previous theorem we can derive the main results of this paper (Theorems \ref{theorem:intro_main1} and \ref{theorem:intro_main2}):
\begin{theorem}
Let $A$ be an algebraically simple, separable, nuclear, $\mathcal{Z}$-stable $C^*$-algebra. Assume that $T(A)$ is a non-empty Bauer simplex and that $\partial_e T(A)$ is finite-dimensional, and let $\alpha\colon \Z \acts A$ be an action on $A$. Then the following hold:
\begin{enumerate}
\item $\alpha \simeq_{\mathrm{cc}} \alpha \otimes \op{id}_\mathcal{Z}$;
\item If $\alpha$ is strongly outer, $\op{dim}^c_{\op{Rok}}(\alpha) \leq 2$. 
\end{enumerate}
\end{theorem}
\begin{proof}
This follows directly from the previous theorem combined with Theorem \ref{theorem:reduction_tracial_property} and \ref{theorem:reduction_rokhlin_dimension}, respectively.
\end{proof}

Moreover, we obtain Corollary \ref{corollary:intro}:
\begin{corollary}
Let $A$ be an algebraically simple, separable, nuclear, $\mathcal{Z}$-stable $C^*$-algebra. Assume that $T(A)$ is a non-empty Bauer simplex and that $\partial_e T(A)$ is finite-dimensional. Let $\mathcal{D}$ be a strongly self-absorbing $C^*$-algebra such that $A \cong A \otimes \mathcal{D}$. Then for any strongly outer action $\alpha\colon \Z \acts A$ and any action $\delta\colon \Z \acts \mathcal{D}$ it holds that $\alpha \simeq_{\mathrm{cc}} \alpha \otimes \delta$.
\end{corollary}
\begin{proof}
The previous theorem implies that $\alpha$ has finite Rokhlin dimension with commuting towers. If $\delta$ is strongly self-absorbing, the result follows directly from \cite[Corollary 5.1]{RokhlinDimension} (using \cite[Theorem 4.10]{SzaboWuZacharias} to verify one additional necessary condition). Just as in the proof of \cite[Corollary 3.5]{ActionsTorsionFree} this can be used to obtain the result for a general $\delta: \Z \acts \mathcal{D}$ as follows. By \cite[Corollary 3.4]{ActionsTorsionFree} there is a unique strongly outer action $\delta^0\colon \Z \acts \mathcal{D}$ up to cocycle conjugacy, so in particular $\delta^0 \simeq_{\mathrm{cc}} \delta^0 \otimes \delta$. Then $\delta^0$ is automatically strongly self-absorbing by \cite[Theorem~3.2(i)]{ActionsTorsionFree},
 hence by the first part of this proof $\alpha \simeq_{\mathrm{cc}} \alpha \otimes \delta^0$. As a consequence we also get $\alpha \simeq_{\mathrm{cc}} \alpha \otimes \delta^0 \simeq_{\mathrm{cc}} \alpha \otimes \delta^0 \otimes \delta \simeq_{\mathrm{cc}} \alpha \otimes \delta$.
\end{proof}

In a similar way, Theorem \ref{theorem:free_case} can be used to prove a result from which Theorem \ref{theorem:intro_free} follows as a special case: 

\begin{theorem}\label{theorem:main_result_free}
Let $A$ be an algebraically simple, separable, nuclear, $\mathcal{Z}$-stable $C^*$-algebra for which the tracial simplex is a non-empty Bauer simplex with finite-dimensional extremal boundary. 
Let $G$ be a countable discrete group for which all finitely generated subgroups are virtually nilpotent. 
Take an action $\alpha\colon G \acts A$ for which the induced action on $\partial_e T(A)$ is free. 
Then for each $n \in \N$ and each unitary representation $\nu\colon G \rightarrow \mathcal{U}(M_n)$ there exists a unital equivariant $*$-homomorphism $(M_n, \op{Ad}(\nu)) \rightarrow (A^\omega \cap A', \alpha^\omega).$ In particular, $\alpha \simeq_{\mathrm{cc}} \alpha \otimes \op{id}_\mathcal{Z}$. 
\end{theorem}
\begin{proof}
 Let $\nu\colon G \rightarrow \mathcal{U}(M_n)$ be a unitary representation. Since $G$ is a countable group, it can be written as an increasing union of finitely generated subgroups. By a reindexation trick similar as in the proof of \cite[Lemma 5.5]{SsaDynSyst3}, it then suffices to prove that for each finitely generated subgroup $H \subset G$ there exists a unital equivariant $*$-homomorphism 
\[(M_n, \op{Ad}(\nu)\big|_{H}) \rightarrow (A^\omega \cap A', \alpha^\omega\big|_{H}).\]
Alternatively, we may just assume that $G$ is finitely generated and virtually nilpotent.

 Similarly as in the proof of Theorem \ref{theorem:main}, we use that $\bar{A}^u$ is isomorphic to the trivial $W^*$-bundle $C_\sigma(\partial_e T(A), \mathcal{R})$. Denote the extension of $\alpha$ to $\bar{A}^u$ by $\gamma$. It follows directly from Theorem \ref{theorem:free_case} that there exists a unital equivariant $*$-homomorphism
\[(M_n, \op{Ad}(\nu)) \rightarrow  ((\bar{A}^u)^\omega \cap (\bar{A}^u)',\gamma^\omega) \cong (A^\omega\cap A', \alpha^\omega).\]This holds in particular when $\nu$ is the trivial representation. Thus, equivariant $\mathcal{Z}$-stability of $\alpha$ follows from Theorem \ref{theorem:reduction_tracial_property}.
\end{proof}

\end{document}